\documentclass[11pt]{article}
\usepackage{amssymb,amsmath,amsthm,amsfonts}
\usepackage[dvips]{graphicx}
\usepackage{graphicx}

\def\disp{\displaystyle}

\def\crr{\cr\noalign{\vskip2mm}}

\theoremstyle{plain}
\newtheorem{theorem}{Theorem}[section]
\newtheorem{lemma}{Lemma}[section]
\newtheorem{proposition}{Proposition}[section]
\newtheorem{corollary}{Corollary}[section]

\numberwithin{equation}{section}

\theoremstyle{definition}

\newtheorem{remark}{Remark}[section]

\begin{document}
\title{{\bf Unique continuation inequalities for  the parabolic-elliptic chemotaxis system\footnote{\small
The authors were supported by the NNSF of China under grants   11971022,  11801408,
and by the Natural Science Foundation of Henan Province  202300410248.
E-mail addresses: wanggs62@yeah.net (Gengsheng Wang); guojiezheng@yeah.net (Guojie Zheng).}} }

\author{Gengsheng Wang$^{a}$, \;   Guojie Zheng$^{b}$
\\
$^a${\it Center for Applied Mathematics}\\
{\it Tianjin University, Tianjin 300072, P.R. China}\\
$^b${\it College of Mathematics and Information Science}\\
{\it Henan Normal University, Xinxiang, 453007, P.R. China}}

\date{}

 \maketitle

\begin{abstract}
This paper studies the quantitative unique continuation  for
 a semi-linear parabolic-elliptic coupled system on a bounded domain $\Omega$.
 This system is a simplified version of the chemotaxis model
 introduced by Keller and Segel in \cite{KellerSegel}.
 With the aid of priori $L^\infty$-estimates (for solutions of the system) built up in this paper, we treat
the semi-linear parabolic equation in the system as a linear parabolic equation, and then use
the frequency function method and the localization technique to build up two unique continuation inequalities for
 the system.
As a consequence of the above-mentioned two inequalities, we have
the following qualitative unique continuation property:
if one component of a solution
 vanishes in a nonempty open subset $\omega\subset\Omega$ at some time $T>0$, then the solution is identically zero.

\vspace{0.3cm}

\vspace{0.3cm}

\noindent {\bf Keywords:}~   Unique continuation inequalities;   parabolic-elliptic coupled system; frequency function method; localization technique

\vspace{0.3cm}

\noindent {\bf AMS subject classifications:}~  35K20, 35K55.
\end{abstract}

\section{Introduction}\label{sec1}
\ \ \ \

 This paper studies the unique continuation  property for  the semi-linear parabolic-elliptic coupled system:
 \begin{equation}\label{1.1}
\begin{cases}
\partial_t u(x,t)-\triangle u(x,t)+\nabla\cdot(u(x,t)\nabla v(x,t))=0,& \textrm{ in } \Omega\times(0,+\infty),\\
-\triangle v(x,t)+av(x,t)-bu(x,t)=0,& \textrm{ in } \Omega\times(0,+\infty),\\
u(x,t)=0, \: v(x,t)=0,&  \textrm{ on } \partial\Omega\times(0,+\infty),\\
u(x,0)=u_0(x),     &   \textrm{ in } \Omega,
\end{cases}
\end{equation}
 where   $\Omega$ is a bounded domain in $\mathbb{R}^n\,\,(n\ge 1)$ with a $C^2$ boundary $\partial\Omega$,
 $a$ and $b$ are  positive constants,  $u_0$ is an initial datum.
 The system \eqref{1.1} is a simplified version of the chemotaxis model
 introduced by Keller and Segel in \cite{KellerSegel}, which
 depicts the change of motion when a population reacts in response to an external chemical stimulus spread in the environment where they reside.  In  the system (\ref{1.1}), $u$ and $v$ stand for the concentration of species and the chemical substance,  respectively.
(See \cite{HillenP}.)  A slightly generalized model of (\ref{1.1})  consists of two parabolic equations:
\begin{eqnarray*}
\begin{cases}
\partial_t u(x,t)=\triangle u(x,t)-\nabla\cdot(u(x,t)\nabla v(x,t)),&  \textrm{ in } \Omega\times(0,+\infty),\\
\tau \partial_t v(x,t) = \triangle v(x,t) - av(x,t) +bu(x,t),&  \textrm{ in } \Omega\times(0,+\infty),
\end{cases}
\end{eqnarray*}
with the time scale $\tau\textless\textless 1$, which means that the time scale of the chemical diffusion is shorter than that of species. (See  \cite{HeVe}.) We refer readers to \cite{HillenP, Horstmann} for more background about  chemotaxis and its model.
In the studies of the chemotaxis model, boundary conditions can be either  Dirichlet type,    Neumann type, or  Robin type, particularly, the homogeneous Dirichlet  boundary condition means the zero density on the boundary. (See  \cite{MLewis}.)  This kind of Keller-Segel system was studied in \cite{DiazNagai, Guo, Jager}, where
the local existence, the global  existence, and the blow-up  phenomenon of solutions were obtained.

  This paper aims to show two unique continuation inequalities
    for the system (\ref{1.1}), from which,
 one can directly get
the following qualitative unique continuation property:
if one component of a solution
 vanishes in a nonempty open subset $\omega\subset\Omega$ at some time $T>0$, then the solution is identically zero.

 Most studies on the unique continuation property for PDEs  focus on the linear cases.
 In  \cite{Landis},  the authors reduced parabolic equations (with constant coefficients)
  to elliptic forms. The  technique developed in  \cite{Landis} also  works  for   parabolic equations with  coefficients depending only on the space-variable.
 The unique continuation property  of parabolic equations with potential was built up in \cite{F. Lin}, in which the order of the solution's vanishing at some interior point was investigated.  We also mention
 \cite{Yamabe} where some weak unique continuation property was obtained. In studies of the quantitative  unique continuation property for linear parabolic equations, the Carleman  inequality and the frequency function are two important tools (see, for instance, \cite{EscauriazaF1}).
 About the  Carleman estimates, we would like to mention works
   \cite{Chaves,  EscauriazaF2, C. Kenig,  H. Koch}.
The frequency function for elliptic equations may be traced back to \cite{Almgren}. A slightly different version of this type of function for parabolic equations was used in \cite{PW} and \cite{Phung3}, where  some   unique continuation inequalities were obtained.
We also mention \cite{Lin2, Kukavica, Phung1,  Poon, Vessella} for the related studies.
  To the best of our knowledge, there are very limited  works on the unique continuation property for nonlinear equations. We mention  \cite{C. Kenig1} and \cite{BZhang} in this direction.  In \cite{BZhang}, the author used the inverse scattering theory to obtain the unique continuation property for the Korteweg-de Vries equation.

The main theorem of this paper is as follows:

\begin{theorem}\label{theorem}
Let  $u_0\in L^\infty(\Omega)$ and let $\omega$ be a nonempty open subset of $\Omega$. Suppose that  $(u,v)$ is the solution to the system (\ref{1.1}) over $[0,T]$ for some $T>0$.
  Then     the following conclusions are true:\\
$(i)$ There are $\gamma= \gamma(\Omega, \omega, \|u_0\|_{L^\infty(\Omega)}, T)\in (0,1)$ and $D=D(\Omega, \omega, \|u_0\|_{L^\infty(\Omega)}, T)>0$
  so that
\begin{eqnarray}\label{1.4}
\int_{\Omega}(|u(x,T)|^2+|v(x,T)|^2)dx\leq D\left(\int_{\Omega}|u_0(x)|^2dx\right)^{1-\gamma}\left(\int_{\omega}|u(x,T)|^2dx\right)^\gamma.
\end{eqnarray}
$(ii)$ When $u_0\neq 0$,  there is   $C=C(\Omega,\omega, \|u_0\|_{L^\infty(\Omega)}, T)>0$ so that
\begin{eqnarray}\label{1.5}
\int_{\Omega}|u_0(x)|^2dx\leq C\exp\left(C{\|u_0(x)\|_{L^2(\Omega)}^2\over \|u_0(x)\|_{H^{-1}(\Omega)}^2}\!\right)
\!\times\! \int_{\omega}(|u(x,T)|^2)dx.
\end{eqnarray}
\end{theorem}
\begin{remark}\label{remark1.2,3-22}
{\it
Several notes on Theorem \ref{theorem}  are given in order.

\begin{itemize}
\item[(a1)] Our motivations to build up Theorem \ref{theorem} are as follows:
From the perspective of mathematics,
 most studies on the unique continuation property  focus on linear systems, while  (\ref{1.1}) is a semi-linear system.
 From the application point of view,
 one can recover the initial population state and evolution history  from the observation in a small  subset at a time, through using our unique continuation inequalities.
  \item[(a2)]  From either (\ref{1.4}) or (\ref{1.5}), we can get the following qualitative unique continuation property
for the system (\ref{1.1}): if $u(x,T)=0$ over $\omega$ or $v(x,T)=0$ over $\omega$,
then $(u,v)$ is identically zero.
  (See Corollary \ref {theorem1.3}.) Moreover, (\ref{1.5}) implies the backward uniqueness: if $u(\cdot,T)=0$
  over $\Omega$, then $u_0=0$.

\item[(a3)] It is natural to ask if (\ref{1.4}) or (\ref{1.5})
 still holds when $\int_{\omega}(|u(x,T)|^2)dx$ is replaced by $\int_{\omega}(|v(x,T)|^2)dx$ on the right hand side.  Unfortunately, we are not able to answer this question  now.

\item[(a4)] Our strategy to show Theorem \ref{theorem} is as follows:
We rewrite
 (\ref{1.1}) as:
\begin{eqnarray}\label{1.6,2-1}
\partial_t u-\triangle u+A\cdot\nabla u +Bu=0,
\end{eqnarray}
where $A=A(v)$ and $B=B(u,v)$ depend on $u$ and $v$. (We can do this because of the structure of (\ref{1.1}).) The equation \eqref{1.6,2-1}
hints us to build up  priori $L^\infty$-estimates for the solution
$(u,v)$ to the system  (\ref{1.1}) with $u_0\in L^\infty(\Omega)$, which gives $L^\infty$-estimates of
$A$ and $B$ in terms of $\|u_0\|_{L^\infty(\Omega)}$. Then we treat \eqref{1.6,2-1} as a linear parabolic equation
where $A$ and $B$ are viewed as coefficients with the above-mentioned $L^\infty$-estimates. After that, we use
 the frequency function method (see, for instance, \cite{AEWZ, Phung1,  PW,  Poon}) to prove a local  interpolation inequality (see Theorem \ref{lemma4.4}). Finally, we utilize the localization technique
 (see \cite{Phung3}) to get Theorem \ref{theorem}.

It seems to us that the above strategy might be used to build up the unique continuation inequalities
for some more general semi-linear parabolic equations.
\end{itemize}}
\end{remark}

The rest of the paper is organized as follows:
Section 2 presents the
well-posedness  and some estimates for the system (\ref{1.1}).
  Section 3 shows a local  interpolation inequality
of  the system (\ref{1.1}). Section 4 proves Theorem \ref{theorem}.


\section{Analysis of the system (\ref{1.1})}

\ \ \ \
We start with introducing notation.
  We use  $\|\cdot\|_p$ (with  $(p\in[1,+\infty])$)  to denote the norm of the space
    $L^p(\Omega)$.  We  use $B(x_0,R)$ to stand for the open ball centered at $x_{0}$ and of radius $R$, and use  $C(\ldots)$ to stand for a positive constant which depends on   what are enclosed in the brackets.

 \subsection{Preliminary lemmas}

 We introduce several estimates on a linear elliptic equation:
 \begin{equation}\label{2.1}
\begin{cases}
-\triangle v(x)+\hat{a}v(x)=\eta(x),& \textrm{ in } \Omega,\\
v(x)=0,&  \textrm{ on } \partial\Omega,\,\end{cases}
\end{equation}
 (Here, $\hat{a}$ is a positive constant and $\eta$ is a given function.)
  and on the semigroup $S_p(t)$ generated by $A_p$ on $L^p(\Omega)$, with $p\in(1,+\infty)$, where
 $$
A_p(u):=\triangle u,\;\; u\in W^{2,p}(\Omega)\cap W^{1,p}_0(\Omega).
$$
The following
two lemmas can be found in \cite{Agmon} and will be used later.
\begin{lemma}\label{lemma2.1}
 For each $p\in (1,+\infty)$ and each  $\eta\in L^{p}(\Omega)$,
 the equation (\ref{2.1}) has  a unique solution $v\in W^{2,p}(\Omega)\cap W_0^{1,p}(\Omega)$. Moreover,
 there is $C=C(\Omega,p,\hat{a})$ so that
 \begin{eqnarray*}
\|v\|_{W^{2,p}(\Omega)}\leq C\|\eta\|_{p}.
 \end{eqnarray*}
 \end{lemma}

\begin{lemma}\label{lemma2.a}
For each $p\in (1,+\infty)$, there is $C=C(\Omega,p)$ so that for each
  $\varphi\in L^p(\Omega)$,
 \begin{eqnarray*}
 \|S_p(t)\varphi\|_{p}\leq \|\varphi\|_{p}   \:\textrm{ and }\: \|\nabla S_p(t)\varphi\|_{p}\leq Ct^{-\frac{1}{2}}\|\varphi\|_{p},\;\;\mbox{when}\;\;t\in(0,+\infty).
 \end{eqnarray*}
 \end{lemma}
 As a direct consequence of Lemma \ref{lemma2.1}, we have
\begin{corollary}\label{remark2.1}
For each $p\in (1,+\infty)$,
the resolvent  $J_{\hat{a}}:=(-\triangle+\hat{a}I)^{-1}$ (with $I$ the identity operator on $L^{p}(\Omega)$)
 is a linear bounded operator from  $L^{p}(\Omega)$ to  $W^{2,p}(\Omega)\cap W_0^{1,p}(\Omega)$. Moreover,
  for each  $\eta\in L^{p}(\Omega)$, the solution of \eqref{2.1} satisfies $v=J_{\hat{a}}(\eta)$.
\end{corollary}

\noindent Though the divergence operator does not commute with $S_p(t)$ $(t>0)$, the next Lemma \ref{lemma2.2}
 remains true.

\begin{lemma}\label{lemma2.2}
For each $p\in (1,+\infty)$,
there is $C=C(\Omega,p)$ so that for each  $\Phi\in [C_0^\infty(\Omega)]^n$,
 \begin{eqnarray}\label{2.2a}
 \|S_p(t)\nabla \cdot\Phi\|_{p}\leq C(\Omega)t^{-\frac{1}{2}}\|\Phi\|_{[L^p(\Omega)]^n},\;\;\mbox{when}\;\;t\in(0,+\infty).
 \end{eqnarray}
\end{lemma}
\begin{proof}
Let $p'\in (1,+\infty)$ satisfy $\frac{1}{p}+\frac{1}{p'}=1$. Then, it follows from Lemma \ref{lemma2.a}
that when $t\in(0,+\infty)$,
 \begin{eqnarray*}
&& \|S_p(t)\nabla \cdot\Phi\|_{p}=\sup_{\|\phi\|_{p'}=1}|\langle S_p(t)\nabla \cdot\Phi, \phi\rangle_{p,p'}|  \nonumber \\
& =&\sup_{\|\phi\|_{p'}=1}|\langle \Phi, \nabla S_{p'}(t)\phi\rangle_{[L^{p}(\Omega)]^n,[L^{p'}(\Omega)]^n}| \leq C(\Omega)t^{-\frac{1}{2}}\|\Phi\|_{[L^p(\Omega)]^n}.
 \end{eqnarray*}
   (Here, $\langle\cdot,\cdot\rangle_{p,p'}$ denotes the pair between $L^{p}(\Omega)$ and $L^{p'}(\Omega)$ and $\langle\cdot,\cdot\rangle_{[L^{p}(\Omega)]^n,[L^{p'}(\Omega)]^n}$ stands for  the pair between $[L^{p}(\Omega)]^n$ and $[L^{p'}(\Omega)]^n$.)
   This completes the proof.
\end{proof}
\begin{remark}
{\it By Lemma \ref{lemma2.2}, we can use the standard density argument
to see that  for each $p\in (1,+\infty)$ and  each $t>0$, the operator $S_p(t)\nabla\cdot$ has a unique extension
 over $[L^p(\Omega)]^n$, which will be denoted in the same manner.
  Thus,  (\ref{2.2a}) holds for all $\Phi\in [L^p(\Omega)]^n$.}
\end{remark}

We end this subsection by introducing an estimate on
 the initial-boundary value problem:
\begin{eqnarray*} \label{2.3.1}
\begin{cases}
\partial_t y(x,t)=\mathcal{A}(t)y(x,t),& \textrm{ in } \Omega\times(0,\bar{T}],\\
y(x,t)=0,&  \textrm{ on } \partial\Omega\times(0,\bar{T}],\\
y(x,0)=y_0(x),  & \textrm{ in } \Omega.
\end{cases}
\end{eqnarray*}
Here,     $\bar{T}>0$ is arbitrarily fixed and   $\mathcal{A}(t)$ is  the differential operator defined by
$$
\mathcal{A}(t)y:=\sum_{i, j = 1}^n{\partial\over\partial{x_i}}\left(a_{ij}(x,t){\partial y\over \partial{x_j}}\right)+\sum_{i=1}^nb_i(x,t){\partial y\over \partial x_i} +c(x,t)y,
$$
 with the real-valued coefficient functions $a_{ij}, b_i, c\in L^\infty(\Omega\times(0,\bar{T}])$, $(i,j=1,\ldots,n)$.
 Moreover, we assume that $\mathcal{A}(t)$ is uniformly strongly elliptic, that is, there exists $\alpha_0>0$ such that for a.e. $(x,t)\in\Omega\times(0,\bar{T}]$,
 \begin{eqnarray}\label{2.2.4}
\sum_{i,j=1}^na_{ij}(x,t)\xi_i\xi_j\geq\alpha_0|\xi|^2\;\;\mbox{for all}\;\;\xi=(\xi_1, \xi_2, \ldots, \xi_n)\in\mathbb{R}^n.
 \end{eqnarray}
\begin{lemma}\label{theorem2.1}
 Let $U(t,s)$ ($\bar{T}\geq t\geq s\geq 0$) be the evolution system generated by $\mathcal{A}(t)$ ($\bar{T}\geq t\geq 0$).
 Then there is a positive constant $\varpi:=\varpi(\Omega,n,L,\alpha_0)$,
where
\begin{eqnarray*}\label{2.10,12.20}
  L=\max\{1, \: \|b_i\|_{L^\infty(\Omega\times(0,\bar{T}))}, \: \|c\|_{L^\infty(\Omega\times(0,\bar{T}))}\:|\: i=1,\ldots,n \},
 \end{eqnarray*}
 and $\alpha_0$ is given by  \eqref{2.2.4},
so that when $1\leq p\leq q \leq+\infty$ and $y_0\in L^p(\Omega)$,
 \begin{eqnarray}\label{2.3.14}
\|U(t,s)y_0\|_{q}\leq e^{\varpi[1+(t-s)]}(t-s)^{-\frac{n}{2}(\frac{1}{p}-\frac{1}{q})}\|y_0\|_{p},
\;\;\mbox{for}\;\; \bar{T}\geq t>s\geq 0.
 \end{eqnarray}
\end{lemma}
\begin{remark} {\it We quote Lemma \ref{theorem2.1} from \cite{Daners},
which doesn't give
what quantities the constants in the inequality \eqref{2.3.14}
depend on. However, following the proof in  \cite{Daners}, we can
get them.}
\end{remark}


 \subsection{An auxiliary system}
In this subsection, we
 study the following  auxiliary system of (\ref{1.1}):
\begin{equation}\label{3.2}
\begin{cases}
\partial_t u(x,t)-\triangle u(x,t)+\nabla\cdot(u(x,t)\nabla v(x,t))=0,& \textrm{ in } \Omega\times(0,+\infty),\crr\disp
-\triangle v(x,t)+av(x,t)-b\xi(x,t)=0,& \textrm{ in } \Omega\times(0,+\infty),\crr\disp
u(x,t)=0, \: v(x,t)=0,&  \textrm{ on } \partial\Omega\times(0,+\infty),\crr\disp
u(x,0)=u_0(x),    &   \textrm{ in } \Omega,
\end{cases}
\end{equation}
where  $\xi\in L^\infty(0,+\infty; L^p(\Omega))$ with $n<p<+\infty$.
\begin{proposition}\label{proposition4.1}
Given $u_0\in L^p(\Omega)$, $\xi\in L^\infty(0,+\infty; L^p(\Omega))$ (with $p\in(n,+\infty)$) and $\bar{T}>0$,
  the
 system (\ref{3.2}) has a unique solution
 $(u,v)$ over $[0,\bar{T}]$. Moreover,
$(u, v)$ belongs to $C([0,\bar{T}]; L^p(\Omega))\times L^\infty(0,\bar{T}; W^{2,p}(\Omega)\cap W_0^{1,p}(\Omega))$ and satisfies that for some  $C:=C\big(\Omega, \bar{T},\|\xi\|_{L^\infty(0,\bar{T}; L^p(\Omega))}\big)>0$,
     \begin{eqnarray}\label{3.2aa}
\|u\|_{C([0,\bar{T}]; L^p(\Omega))}
\leq C\|u_0\|_{p}.
 \end{eqnarray}
 If we further assume that $\xi\in C([0, +\infty); L^p(\Omega))$, then $(u,v)\in C([0,\bar{T}]; L^p(\Omega))\times C((0,\bar{T}]; W^{2,p}(\Omega)\cap W_0^{1,p}(\Omega))$.
\end{proposition}
\begin{proof}
Arbitrarily fix $p\in(n,+\infty)$, $u_0\in L^p(\Omega)$, $\xi\in L^\infty(0,+\infty; L^p(\Omega))$
and  $\bar{T}>0$. Two observations are given in order.
First, according to Corollary \ref{remark2.1},
 \begin{eqnarray}\label{2.8}
  v(t)=J_{a}\big(b\xi(t)\big),\:\: t\in(0,\bar{T}],
 \end{eqnarray}
and  $J_{a}\in \mathcal{L}(L^{p}(\Omega); W^{2,p}(\Omega)\cap W_0^{1,p}(\Omega))$,
where $v$ is the unique solution to the second equation of (\ref{3.2}).
 From these and the fact that
 $\xi\in L^\infty(0,+\infty; L^p(\Omega))$, it follows from the standard method in \cite{Yosida} that
 $v\in L^\infty(0,\bar{T}; W^{2,p}(\Omega)\cap W_0^{1,p}(\Omega))$. (See  Section V of \cite{Yosida},  p.134.)
Then by  the Sobolev embedding theorem and Lemma \ref{lemma2.1}, we have
 \begin{eqnarray}\label{3.3}
  &&\|v\|_{ L^\infty(0,\bar{T}; L^\infty(\Omega))}+\|\nabla v\|_{ L^\infty(0,\bar{T}; [L^\infty(\Omega)]^n)}\crr\disp
  &\leq& C(\Omega) \|v\|_{ L^\infty(0,\bar{T}; W^{2,p}(\Omega)\cap W_0^{1,p}(\Omega))}\leq C(\Omega)\|\xi\|_{L^\infty(0,\bar{T}; L^p(\Omega))}.
 \end{eqnarray}

 Second, we arbitrarily fix
   $\eta\in C([0,\bar{T}]; L^p(\Omega))$.  Then it follows from Lemma \ref{lemma2.a} and  Lemma \ref{lemma2.2} that
    for each $t\in[0,\bar{T}]$,
 \begin{eqnarray}\label{2.10,2-11}
&&\|S_p(t)u_0-\int_0^t S_p(t-s)\nabla\cdot(\eta \nabla v)ds\|_{p}\crr\disp
&\leq &\|u_0\|_{p}+C(\Omega) \int_0^t(t-s)^{-\frac{1}{2}}\|\eta\nabla v\|_{[L^p(\Omega)]^n}ds\crr\disp
&\leq& \|u_0\|_{p}+C(\Omega) \bar{T}^{\frac{1}{2}}\|\eta\|_{C([0,\bar{T}]; L^p(\Omega))}\|\nabla v\|_{L^\infty(0,\bar{T};[L^\infty(\Omega)]^n)}.
 \end{eqnarray}
Now by \eqref{2.10,2-11} and  (\ref{3.3}), we see that for each $t\in[0,\bar{T}]$,
 \begin{eqnarray*}
\|S_p(t)u_0-\int_0^t S_p(t-s)\nabla\cdot(\eta \nabla v)ds\|_{p}
\leq \|u_0\|_{p}+C(\Omega) \bar{T}^{\frac{1}{2}}\|\eta\|_{C([0,\bar{T}]; L^p(\Omega))}\|\xi\|_{L^\infty(0,\bar{T}; L^p(\Omega))}.
 \end{eqnarray*}
  This yields $F(\cdot;\eta)\in C([0,\bar{T}]; L^p(\Omega))$, where
  \begin{eqnarray*}
  F(t;\eta):=S_p(t)u_0-\int_0^t S_p(t-s)\nabla\cdot(\eta\nabla v)ds,\;\;t\in [0,\bar{T}].
  \end{eqnarray*}
  Thus, we can define a map $\Lambda: C([0,\bar{T}]; L^p(\Omega))\mapsto C([0,\bar{T}]; L^p(\Omega))$
  in the following manner: for each $\eta\in C([0,\bar{T}]; L^p(\Omega))$, set
  \begin{eqnarray*}
   \Lambda(\eta)(t):=F(t;\eta),\; t\in [0,\bar{T}].
  \end{eqnarray*}

  We are going to finish the proof with the aid of the contraction mapping theorem.
  To this end, we
   see from  Lemma \ref{lemma2.2} that when  $\eta_1, \eta_2\in C([0,\bar{T}]; L^p(\Omega))$,
  \begin{eqnarray*}
\|\Lambda(\eta_1)-\Lambda(\eta_2)\|_{C([0,\bar{T}]; L^p(\Omega))}&\leq&\int_0^t \|S_p(t-s)\nabla\cdot((\eta_1-\eta_2)\nabla v)\|ds\crr\disp
&\leq&\int_0^t C(\Omega)(t-s)^{-\frac{1}{2}}\|(\eta_1-\eta_2)\nabla v\|_{[L^p(\Omega)]^n}ds.
 \end{eqnarray*}
  This, along with (\ref{3.3}), implies
  \begin{eqnarray}\label{3.4aa}
\|\Lambda(\eta_1)-\Lambda(\eta_2)\|_{C([0,\bar{T}]; L^p(\Omega))}
\leq C(\Omega) \bar{T}^{\frac{1}{2}}\|\xi\|_{L^\infty(0,\bar{T}; L^p(\Omega))}\|\eta_1-\eta_2\|_{C([0,\bar{T}]; L^p(\Omega))}.
 \end{eqnarray}

  We first consider the case that
  \begin{eqnarray}\label{2.12-2-11}
  C(\Omega) \bar{T}^{\frac{1}{2}}\|\xi\|_{L^\infty(0,\bar{T}; L^p(\Omega))}<1.
  \end{eqnarray}
  From \eqref{3.4aa} and \eqref{2.12-2-11}, we see that
  $\Lambda$ is a strict contraction map. Thus it  has a unique fixed point $u\in C([0,\bar{T}]; L^p(\Omega))$, i.e.,
   \begin{eqnarray*}
u(t)=S_p(t)u_0-\int_0^t S_p(t-s)\nabla\cdot(u\nabla v)ds,\;\;t\in [0,\bar{T}].
 \end{eqnarray*}
  Consequently, $u\in C([0,\bar{T}]; L^p(\Omega))$ is the unique solution to the first equation of (\ref{3.2}) (corresponding to the above $v$) over $[0,\bar{T}]$. Hence,
  $(u,v)\in C([0,\bar{T}]; L^p(\Omega))\times L^\infty(0,\bar{T}; W^{2,p}(\Omega)\cap W_0^{1,p}(\Omega))$
 is the unique solution to (\ref{3.2}) over $[0,\bar{T}]$ in the case \eqref{2.12-2-11}.
    Meanwhile, one can easily  check that  $\Lambda(0)=S_p(t)u_0\in C([0,\bar{T}]; L^p(\Omega))$.
 By these, we can take  $\eta_1=u$ and $\eta_2=0$ in (\ref{3.4aa}) to get
 \begin{eqnarray*}
&&\|u\|_{C([0,\bar{T}]; L^p(\Omega))}=\|\Lambda(u)\|_{C([0,\bar{T}]; L^p(\Omega))} \nonumber \\
&\leq& \|S_p(t)u_0\|_{C([0,\bar{T}]; L^p(\Omega))}+\|\Lambda(u)-\Lambda(0)\|_{C([0,\bar{T}]; L^p(\Omega))}  \nonumber \\
&\leq&  \|u_0\|_{p}+ C(\Omega) \bar{T}^{\frac{1}{2}}\|\xi\|_{L^\infty(0,\bar{T}; L^p(\Omega))}\|u\|_{C([0,\bar{T}]; L^p(\Omega))},
 \end{eqnarray*}
which leads to \eqref{3.2aa} for the case  \eqref{2.12-2-11}.

We next consider the case that
\begin{eqnarray}\label{2.13,2-11}
C(\Omega) \bar{T}^{\frac{1}{2}}\|\xi\|_{L^\infty(0,\bar{T}; L^p(\Omega))}\geq1.
\end{eqnarray}
By using a  standard iteration argument (by choosing another $\bar{T}_1>0$ with $C(\Omega) \bar{T}_1^{\frac{1}{2}}\|\xi\|_{L^\infty(0,\bar{T}; L^p(\Omega))}<1$) and the results obtained in the case
\eqref{2.12-2-11}, we can get the desired results for the case \eqref{2.13,2-11}.

Finally, we assume that $\xi\in C([0, +\infty); L^p(\Omega))$. Then
it follows from \eqref{2.8}, the linearity and boundedness of  $J_{a}$ that
  \begin{eqnarray*}
 && \|v(t_1)-v(t_2)\|_{W^{2,p}(\Omega)}= \|bJ_{a}\big(\xi(t_1)-\xi(t_2)\big)\|_{W^{2,p}(\Omega)}\crr\disp
  &\leq &b\|J_{a}\|\cdot\|\xi(t_1)-\xi(t_2)\|_{p},\;\;\;\;\mbox{when}\;\;t_1, t_2\in(0,+\infty).
 \end{eqnarray*}
This, along with the continuity of $\xi$, yields that $(u,v)\in C([0,\bar{T}]; L^p(\Omega))\times C((0,\bar{T}]; W^{2,p}(\Omega)\cap W_0^{1,p}(\Omega))$.

Hence, we  finish the proof.
\end{proof}

\subsection{Well-posedness and estimates on  the system (\ref{1.1})}
In this subsection, we  first build up the well-posedness for the system (\ref{1.1}), then we present
  some  estimates for the solution to the system (\ref{1.1}).
\begin{theorem}\label{theorem4.2}
Let $n<p<+\infty$ and  $u_0\in L^p(\Omega)$. Then  the following conclusions are true:\\
$(i)$  The system (\ref{1.1}) has a unique solution  $(u, v)\in C([0,T^*]; L^p(\Omega))\times C((0,T^*]; W^{2,p}(\Omega)\cap W_0^{1,p}(\Omega))$ for some
positive constant $T^*:=T^*(\|u_0\|_p, \Omega)$.\\
$(ii)$ The above solution satisfies $(t^{\frac{n}{2p}}u,v)\in [L^\infty(0,T^*; L^\infty(\Omega))]^2$.\\
 $(iii)$ If we further assume that $u_0\in L^\infty(\Omega)$, then $(u,v)\in[L^\infty(0,T^*; L^\infty(\Omega))]^2$.
\end{theorem}
\begin{proof}
Arbitrarily fix $n<p<+\infty$ and  $u_0\in L^p(\Omega)$. We organize the proof in several steps.

\noindent{\it Step 1. We prove
that for some $T^*:=T^*(\|u_0\|_p, \Omega)>0$, the system (\ref{1.1}) has a  solution  $(u, v)\in C([0,T^*]; L^p(\Omega))\times C((0,T^*]; W^{2,p}(\Omega)\cap W_0^{1,p}(\Omega))$ satisfying
\begin{eqnarray}\label{3.6aa}
\|u\|_{C([0,T^*]; L^p(\Omega))}\leq C(\|u_0\|_{p}, T^*, \Omega).
 \end{eqnarray}
}

We will use the contraction mapping theorem to prove it.
To this end, we first set up the following framework: Let  $T^*>0$, which will be determined later.
Let
\begin{eqnarray*}
K:=\{\xi\in C([0,T^*]; L^p(\Omega)) \:\:   |  \:\:\:\|\xi\|_{C([0,T^*]; L^p(\Omega))}\leq 2\|u_0\|_{p}\}.
\end{eqnarray*}
According to Proposition \ref{proposition4.1}, for each $\xi\in K$,
the system (\ref{3.2}) (with the above $u_0$ and the zero extension of $\xi$ over $(T^*,+\infty)$)
has a unique solution $(u^\xi,v^\xi)\in C([0,T^*]; L^p(\Omega))\times C((0,T^*]; W^{2,p}(\Omega)\cap W_0^{1,p}(\Omega))$.
We define a map $\Psi$ from $K$ to $C([0,T^*]; L^p(\Omega))$ by
setting
\begin{eqnarray*}
\Psi(\xi):=u^\xi\;\;\mbox{for each}\;\; \xi\in K.
\end{eqnarray*}
We claim that for some $T^*>0$, $\Psi$ has a unique fixed point $u$ in $K$.
When this is done, we have
\begin{eqnarray*}
u(t)=S_p(t)u_0-\int_0^t S_p(t-s)\nabla\cdot(u\nabla v)ds,\; t\in[0,T^*]
 \end{eqnarray*}
 and
  \begin{eqnarray*}
\begin{cases}
-\triangle v+av-bu=0,& \textrm{ in } \Omega\times(0,T^*],\\
 v=0,&  \textrm{ on } \partial\Omega\times(0,T^*].
\end{cases}
\end{eqnarray*}
From these, we see that $(u,v):=(u,v^u)\in C([0,T^*]; L^p(\Omega))\times C((0,T^*]; W^{2,p}(\Omega)\cap W_0^{1,p}(\Omega))$ is a solution to the system (\ref{1.1}).

We now check conditions ensuring the contraction mapping theorem.
First, it is clear that  $K$ is a closed subset of  $C([0,T^*]; L^p(\Omega))$.
Second, we claim that $\Psi(K)\subseteq K$ for some $T^*>0$. Indeed, it  follows from Lemma \ref{lemma2.a} and
Lemma \ref{lemma2.2}  that
\begin{eqnarray*}
&&\|u^\xi(t)\|_{p}\leq\|S_p(t)u_0\|_{p}+\|\int_0^t S_p(t-s)\nabla\cdot(u^\xi\nabla v^\xi)ds\|_{p}\crr\disp
&\leq &\|u_0\|_{p}+ C(\Omega)t^{\frac{1}{2}}\|\nabla v^\xi\|_{L^\infty(0,T^*; [L^\infty(\Omega)]^n)}\|u^\xi\|_{C([0,T^*]; L^p(\Omega))}
\;\;\mbox{for each}\;\;t\in[0, T^*].
 \end{eqnarray*}
This, together with (\ref{3.3}) and the fact $\xi\in K$, yields
 \begin{eqnarray*}
&&\|u^\xi(t)\|_{p}\leq \|u_0\|_{p}+ C(\Omega)t^{\frac{1}{2}}\|\xi\|_{C([0,T^*]; L^p(\Omega))}\|u^\xi\|_{C([0,T^*]; L^p(\Omega))}\crr\disp
&\leq &\|u_0\|_{p}+ C(\Omega)T^{*\frac{1}{2}}\|u_0\|_{p}\|u^\xi\|_{C([0,T^*]; L^p(\Omega))}
\;\;\mbox{for each}\;\;t\in[0, T^*],
 \end{eqnarray*}
 which leads to
  \begin{eqnarray}\label{3.5,12.25}
\|u^\xi\|_{C([0,T^*]; L^p(\Omega))}\leq \frac{1}{1-C(\Omega)T^{*\frac{1}{2}}\|u_0\|_{p}}\|u_0\|_{p}.
 \end{eqnarray}
By choosing $T^*:=T^*(\|u_0\|_p,\Omega)$  so that
  \begin{eqnarray}\label{3.5aa}
\frac{1}{1-C(\Omega)T^{*\frac{1}{2}}\|u_0\|_{p}}\leq2,
 \end{eqnarray}
 we get from \eqref{3.5,12.25} that $\Psi(\xi)\in K$. Therefore, $\Psi(K)\subseteq K$.
Third, we claim that $\Psi$ is a strict contraction for some $T^*>0$ satisfying \eqref{3.5aa}. In fact, given $\xi_1,\xi_2\in K$,
 write $u_i=\Psi(\xi_i)$ and $v_i=v^{\xi_i}$ ($i=1,2$). Let
 $w:=u_1-u_2$.
 Then one can easily check that
   $w$ belongs to $C([0,T^*]; L^p(\Omega))$ and satisfies
 \begin{eqnarray*}
\begin{cases}
\partial_t w-\triangle w+\nabla\cdot(w\nabla v_1)+\nabla\cdot(u_2\nabla(v_1-v_2))=0,& \textrm{ in } \Omega\times(0,T^*],\\
w=0, &  \textrm{ on } \partial\Omega\times(0,T^*],\\
w(x,0)=0, & \textrm{ in } \Omega.
\end{cases}
\end{eqnarray*}
Thus, we have
 \begin{eqnarray*}
w(t)=-\int_0^tS_p(t-s)\nabla\cdot[ w\nabla v_1+u_2\nabla(v_1-v_2)]ds,\;\;t\in [0,T^*].
 \end{eqnarray*}
This, along with   Lemma \ref{lemma2.2},
yields that when  $t\in[0,T^*]$,
 \begin{eqnarray*}
&&\|w(t)\|_{p}\leq C(\Omega)\int_0^t(t-s)^{-\frac{1}{2}}\big(\|\nabla v_1\|_{[L^\infty(\Omega)]^n}\|w\|_{p}+\|\nabla(v_1-v_2)\|_{[L^\infty(\Omega)]^n}\|u_2\|_{p}\big)ds\crr\disp
&\leq& \!\!\!\! C(\Omega)T^{*\frac{1}{2}}\big(\|\nabla v_1\|_{L^\infty(0,T^*; [L^\infty(\Omega)]^n)}\|w\|_{C([0,T^*]; L^p(\Omega))}\!+\!\|\nabla(v_1-v_2)\|_{L^\infty(0,T^*;[L^\infty(\Omega)]^n)}\|u_2\|_{C([0,T^*]; L^p(\Omega))}\big).
 \end{eqnarray*}
 The above inequality, together with (\ref{3.3})  and the fact that $\xi_1, u_2\in K$, indicates that when  $t\in[0,T^*]$,
  \begin{eqnarray*}
  \|w(t)\|_{p}\leq  C(\Omega)T^{*\frac{1}{2}}\|u_0\|_{p}\big(\|w\|_{C([0,T^*]; L^p(\Omega))}+\|\xi_1-\xi_2\|_{C([0,T^*]; L^p(\Omega))}\big),
 \end{eqnarray*}
from which it follows that
  \begin{eqnarray}\label{3.6bb}
\|\Psi(\xi_1)-\Psi(\xi_2)\|_{C([0,T^*]; L^p(\Omega))}\leq \frac{C(\Omega)T^{*\frac{1}{2}}\|u_0\|_{p}}{1-C(\Omega)T^{*\frac{1}{2}}\|u_0\|_{p}}\|\xi_1-\xi_2\|_{C([0,T^*]; L^p(\Omega))}.
 \end{eqnarray}
Choosing $T^*:=T^*(\|u_0\|_p,\Omega)$ satisfying \eqref{3.5aa} and
  \begin{eqnarray*}
  \frac{C(\Omega)T^{*\frac{1}{2}}\|u_0\|_{p}}{1-C(\Omega)T^{*\frac{1}{2}}\|u_0\|_{p}}<1,
 \end{eqnarray*}
we see from  \eqref{3.6bb} that
 $\Psi$ is a strict contraction map from $K$ to $K$.
 Thus, according to the  contraction mapping theorem, $\Psi$ has a unique fixed point $u$ in $K$.

 Finally, we show \eqref{3.6aa}.
 Indeed,  one can easily  check that  $\Psi(0)=S_p(t)u_0\in K$. Then
 by  taking $\xi_1=u$ and $\xi_2=0$ in (\ref{3.6bb}), we have
  \begin{eqnarray*}
&&\|u\|_{C([0,T^*]; L^p(\Omega))}=\|\Psi(u)\|_{C([0,T^*]; L^p(\Omega))} \nonumber \\
&\leq& \|\Psi(0)\|_{C([0,T^*]; L^p(\Omega))}+\|\Psi(u)-\Psi(0)\|_{C([0,T^*]; L^p(\Omega))}  \nonumber \\
&\leq&  \|u_0\|_{p}+ \frac{C(\Omega)T^{*\frac{1}{2}}\|u_0\|_{p}}{1-C(\Omega)T^{*\frac{1}{2}}\|u_0\|_{p}}\|u\|_{C([0,T^*]; L^p(\Omega))},
 \end{eqnarray*}
which leads to \eqref{3.6aa}.

\noindent{\it Step 2. We
show that the solution of  the system (\ref{1.1}) is unique.}

By contradiction, we suppose that it is not true. Then
we could find another  solution $(\tilde{u}, \tilde{v})\in$ $C([0,T^*]; L^p(\Omega))\times C((0,T^*]; W^{2,p}(\Omega)\cap W_0^{1,p}(\Omega))$ to the system  (\ref{1.1}) (with the initial datum $u_0$),
 differing from the solution $(u,v)$ obtained in Step 1.
Write
  \begin{eqnarray}\label{3.10bb}
\varrho:=\max \{\|u\|_{C([0,T^*]; L^p(\Omega))}, \|\tilde{u}\|_{C([0,T^*]; L^p(\Omega))} \},
 \end{eqnarray}
and
  \begin{eqnarray}\label{3.11aa}
t_0:=\inf\{t\in[0,T^*]\:|\: u(t)\neq\tilde{u}(t)\}.
 \end{eqnarray}
It is obvious that $0\leq t_0<T^*$.
By the continuity  of $u$ and $\tilde{u}$, we have $u(t_0)=\tilde{u}(t_0)$. Now, we have
  \begin{eqnarray*}
u(t)-\tilde{u}(t)&=&[S_p(t-t_0)u(t_0)-\int_{t_0}^t S_p(t-s)\nabla\cdot(u\nabla v)ds]\\
&-&[S_p(t-t_0)\tilde{u}(t_0)-\int_{t_0}^t S_p(t-s)\nabla\cdot(\tilde{u}\nabla \tilde{v})ds],\;\;t\in(t_0, T^*].
 \end{eqnarray*}
After simple computations, we see
  \begin{eqnarray*}
&&\|u(t)-\tilde{u}(t)\|_p=\|\int_{t_0}^t S_p(t-s)\nabla\cdot(u\nabla v)ds-\int_{t_0}^t S_p(t-s)\nabla\cdot(\tilde{u}\nabla \tilde{v})ds\|_p\\
 &\leq&\int_{t_0}^t \|S_p(t-s)\nabla\cdot((u-\tilde{u})\nabla v)+S_p(t-s)\nabla\cdot(\tilde{u}\nabla (v-\tilde{v}))\|_pds,\;\;t\in(t_0, T^*].
 \end{eqnarray*}
This, along with   Lemma \ref{lemma2.2}, yields that for each $t\in(t_0, T^*]$,
  \begin{eqnarray}\label{2.17}
&&\|u(t)-\tilde{u}(t)\|_p\leq C(\Omega)\int_{t_0}^t (t-s)^{-\frac{1}{2}}(\|(u-\tilde{u})\nabla v\|_p+ \|\tilde{u}\nabla (v-\tilde{v})\|_p)ds \crr\disp
&\leq& C(\Omega)\int_{t_0}^t(t-s)^{-\frac{1}{2}}\big(\|\nabla v\|_{[L^\infty(\Omega)]^n}\|u-\tilde{u}\|_{p}+\|\nabla(v-\tilde{v})\|_{[L^\infty(\Omega)]^n}
\|\tilde{u}\|_{p}\big)ds.
 \end{eqnarray}
By  the second equation of the system \eqref{1.1}, Lemma \ref{lemma2.1}, and the Sobolev embedding theorem, we have that for each $s\in(0,T^*]$,
   \begin{eqnarray}\label{2.18}
\|\nabla v(s)\|_{[L^\infty(\Omega)]^n}\leq C(\Omega)\|u(s)\|_{p} \;\; \textrm{  and  }\;\;
 \|\nabla(v-\tilde{v})(s)\|_{[L^\infty(\Omega)]^n}\leq C(\Omega)\|(u-\tilde{u})(s)\|_{p}.
 \end{eqnarray}
 This, along with \eqref{2.17} and (\ref{3.10bb}), implies
 \begin{eqnarray*}
\|u(t)-\tilde{u}(t)\|_p&\leq &C(\Omega)\int_{t_0}^t(t-s)^{-\frac{1}{2}}\big(\|u\|_{p}\|u-\tilde{u}\|_{p}+\|u-\tilde{u}\|_{p}
\|\tilde{u}\|_{p}\big)ds\crr\disp
 &\leq& 4\varrho C(\Omega)(t-t_0)^{\frac{1}{2}}\|u-\tilde{u}\|_{C([t_0,t];L^p(\Omega))},\;\;t\in(t_0,T^*],
 \end{eqnarray*}
which leads to
\begin{eqnarray*}
\|u-\tilde{u}\|_{C([t_0,t];L^p(\Omega))}\leq4\varrho C(\Omega)(t-t_0)^{\frac{1}{2}}\|u-\tilde{u}\|_{C([t_0,t];L^p(\Omega))}, \;\; t\in(t_0,T^*].
\end{eqnarray*}
Let $\epsilon>0$ satisfy that
 $4\varrho C(\Omega)\epsilon^{\frac{1}{2}}<1$ and $t_0+\epsilon\leq T^*$. Then the above inequality  indicates that
  \begin{eqnarray*}
u(s)=\tilde{u}(s), \textrm{ when } s\in(t_0,t_0+\epsilon),
\end{eqnarray*}
which leads to a contradiction to (\ref{3.11aa}).
Hence, the desired uniqueness holds.

\noindent{\it Step 3. We show (ii), with the aid of Lemma \ref{theorem2.1}.}

By the same argument in the proof of \eqref{2.18}, we can easily check that the solution $(u,v)$ obtained in Step 1 satisfies
   \begin{eqnarray}\label{2.19}
\|v\|_{L^\infty(0,T^*; L^\infty(\Omega))}+\|\nabla v\|_{L^\infty(0,T^*; [L^\infty(\Omega)]^n)}\leq C(\Omega)\|u\|_{C([0,T^*]; L^p(\Omega))}.
 \end{eqnarray}
This, along with (\ref{3.6aa}),  implies that
  \begin{eqnarray*}
\|v\|_{L^\infty(0,T^*; L^\infty(\Omega))}+\|\nabla v\|_{L^\infty(0,T^*; [L^\infty(\Omega)]^n)} \leq C(\|u_0\|_{p}, T^*, \Omega),
 \end{eqnarray*}
which leads to  $v\in L^\infty(0,T^*; L^\infty(\Omega))$.

Next, we define a differential operator $\mathcal{A}_1(t)$ by
$$
\mathcal{A}_1(t)\varphi:=\triangle \varphi-\nabla v\cdot\nabla \varphi -av\varphi,\;\;\varphi\in W^{2,p}(\Omega)\cap W^{1,p}_0(\Omega),
$$
and let $U_1(t,s)$ be the evolution system generated by  $\mathcal{A}_1(t)$ $(T^*\geq t\geq0)$.
Since  (\ref{1.1}) can be rewritten as:
\begin{eqnarray*}
\begin{cases}
\partial_t u-\triangle u+\nabla v\cdot\nabla u +a vu=b u^2,& \textrm{ in } \Omega\times(0,T^*],\\
-\triangle v+av-bu=0,& \textrm{ in } \Omega\times(0,T^*],\\
u=0, \: v=0,&  \textrm{ on } \partial\Omega\times(0,T^*],\\
u(x,0)=u_0,     &   \textrm{ in } \Omega,
\end{cases}
\end{eqnarray*}
we see from the variation of constants formula that
 \begin{eqnarray}\label{3.8}
u(t)=U_1(t,0)u_0+b\int_0^tU_1(t,s)u^2ds,\;\;   t\in[0,T^*].
 \end{eqnarray}
Then, it follows from (\ref{3.8}) and Lemma \ref{theorem2.1} that
 \begin{eqnarray*}
&&\|u(t)\|_{\infty}\leq\|U_1(t,0)u_0\|_{\infty}+\|b\int_0^tU_1(t,s)u^2(s) ds\|_{\infty}\crr\disp
&\leq &e^{\varpi(1+t)}t^{-\frac{n}{2p}}\|u_0\|_{p}+b\int_0^t(t-s)^{-\frac{n}{p}}e^{\varpi[1+(t-s)]}
\|u^2(s)\|_{\frac{p}{2}}ds,\;\;   t\in(0,T^*],
 \end{eqnarray*}
where $\varpi$ is a positive number depending on $C(\|u_0\|_{p}, T^*, n, \Omega)$. After some computations, we obtain
  \begin{eqnarray*}
&&\|u(t)\|_{\infty}\leq e^{\varpi(1+t)}t^{-\frac{n}{2p}}\|u_0\|_{p}+b\int_0^t(t-s)^{-\frac{n}{p}}e^{\varpi[1+(t-s)]}\|u(s)\|^2_{p}ds\crr\disp
&\leq & e^{\varpi(1+T^*)}\left(t^{-\frac{n}{2p}}\|u_0\|_{p}+ \frac{bp}{p-n}t^{1-\frac{n}{p}}\|u\|_{C([0,T^*]; L^p(\Omega))}^2\right),\;\;   t\in(0,T^*].
 \end{eqnarray*}
  Since $p>n$, the above inequality leads to $t^{\frac{n}{2p}}u\in L^\infty(0,T^*; L^\infty(\Omega))$. In summary, we conclude that $(t^{\frac{n}{2p}}u,v)\in L^\infty(0,T^*; L^\infty(\Omega))]^2$.

\noindent{\it Step 4. We show (iii).}

Let $u_0\in L^\infty(\Omega)$. By Step 1 and Step 3, we already have
$v\in L^\infty(0,T^*; L^\infty(\Omega))$ and  $u\in C([0,T^*]; L^p(\Omega))$ $(n<p<+\infty)$.
The remainder is to show $u\in L^\infty(0,T^*; L^\infty(\Omega))$.
To this end, we see from  (\ref{3.8})  and Lemma \ref{theorem2.1} that
  \begin{eqnarray}\label{3.10}
&&\|u(t)\|_{\infty}\leq\|U_1(t,0)u_0\|_{\infty}+\|b\int_0^tU_1(t,s)u^2(s) ds\|_{\infty}\crr\disp
&\leq &e^{\varpi(1+t)}\|u_0\|_{\infty}+b\int_0^t(t-s)^{-\frac{n}{p}}e^{\varpi[1+(t-s)]}\|u(s)\|^2_{p}ds\crr\disp
&\leq &e^{\varpi(1+T^*)}\big(\|u_0\|_{\infty}+ \frac{bp}{p-n}t^{1-\frac{n}{p}}(\|u\|_{C([0,T^*]; L^p(\Omega))}^2\big)<+\infty,\;\;t\in[0,T^*],
 \end{eqnarray}
where $\varpi$ is a positive number depending on $C(\|u_0\|_{\infty}, T^*, n, \Omega)$.
 From \eqref{3.10} and $p>n$, it follows that $u\in L^\infty(0,T^*; L^\infty(\Omega))$.

Hence, we complete  the proof of Theorem \ref{theorem4.2}.
\end{proof}

\begin{remark}\label{remark2.3,2-11}
{\it
$(i)$ It follows from (\ref{3.6aa}) and (\ref{3.10}) that when $u_0\in L^\infty(\Omega)$,
 \begin{eqnarray*}
\|u\|_{L^\infty(0,T^*; L^\infty(\Omega))}\leq  C(\|u_0\|_\infty, T^*, n, \Omega).
 \end{eqnarray*}

 $(ii)$  Let  $u_0\in L^\infty(\Omega)$. Suppose that $(u,v)$ is the solution to the system (\ref{1.1}) over $[0,T]$ for some $T>0$.
 First, by the conclusion $(iii)$ in Theorem \ref{theorem4.2},  we have  that
 $(u,v)\in[L^\infty(0,T; L^\infty(\Omega))]^2$,  and it follows from  (\ref{2.19}) that
 \begin{eqnarray}\label{4.3bb}
\|v\|_{ L^\infty(0,T; L^\infty(\Omega))}\leq CM \textrm{ and } \|\nabla v\|_{ L^\infty(0,T; [L^\infty(\Omega)]^n)}\leq CM,
\end{eqnarray}
  where $C>0$ depends only  on $\Omega$ and where
 \begin{eqnarray}\label{M}
M:=M(\|u_0\|_\infty, T):=\max\{1, \|u\|_{L^\infty(0, T; L^{\infty}(\Omega))}\}.
 \end{eqnarray}
 Second, for each $p\in(n, +\infty)$,
 the above solution  $(u, v)\in C([0,T]; L^p(\Omega))\times C((0,T]; W^{2,p}(\Omega)\cap W_0^{1,p}(\Omega))$.

   $(iii)$ By \eqref{1.1} and the uniqueness given in Theorem \ref{theorem4.2}, we
can easily verify what follows:  $u_0=0$
over $\Omega$ if and only if $u=v=0$ over $\Omega\times (0,+\infty)$.}
\end{remark}

Next, we  present some  estimates for  the system (\ref{1.1}).
 \begin{theorem}\label{proposition3.2}
 Let  $u_0\in L^\infty(\Omega)$. Suppose that  $(u,v)$ is the solution to the system (\ref{1.1}) over $[0,T]$ for some $T>0$.
 Then,
 \begin{eqnarray}\label{3.8aaa}
\|u(\cdot, t)\|_2\leq e^{L_1M^2t}\|u_0\|_2, \;\; t\in[0,T];
 \end{eqnarray}
\begin{eqnarray}\label{3.9aaa}
\|\nabla u(\cdot, t)\|_2\leq \frac{e^{L_1M^2t}}{\sqrt{t}}\|u_0\|_2, \;\; t\in(0,T],
 \end{eqnarray}
 where  $L_1>0$ is a constant  depending only on $\Omega$ and where $M$ is  given by (\ref{M}).
\end{theorem}
\begin{proof}
Multiplying  the first equation of (\ref{1.1}) by $e^{-lt}\partial_t u$, where $l>0$  will be determined later, and then integrating  it over $\Omega$, we obtain
\begin{eqnarray*}
\int_{\Omega}[\partial_t u-\triangle u+\nabla\cdot(u\nabla v)]\cdot e^{-lt}\partial_t udx=0,\; t\in[0,T].
 \end{eqnarray*}
 From the above,  the second equation of the system \eqref{1.1}, \eqref{4.3bb},  Cauchy's inequality, and Lemma \ref{lemma2.1},  we find
\begin{eqnarray*}
&& e^{-lt}\int_{\Omega}|\partial_t u|^2dx + \frac{1}{2}\frac{d}{dt}(e^{-lt}\int_{\Omega}|\nabla u|^2dx) \crr\disp
&=& -e^{-lt}\int_{\Omega}(\nabla u \cdot\nabla v+auv-bu^2)\cdot\partial_t udx-\frac{l}{2}e^{-lt}\int_{\Omega}|\nabla u|^2dx \crr\disp
&\leq& \frac{1}{2}e^{-lt}\int_{\Omega}|\partial_t u|^2dx+ C(\Omega)M^2e^{-lt}\int_{\Omega}|\nabla u|^2dx \crr\disp
&+& C(\Omega)M^2e^{-lt}\int_{\Omega}|u|^2dx+\frac{3b^2}{2}M^2e^{-lt}
\int_{\Omega}|u|^2dx-\frac{l}{2}e^{-lt}\int_{\Omega}|\nabla u|^2dx,\; t\in [0,T],
 \end{eqnarray*}
which yields that when $t\in[0,T]$,
 \begin{eqnarray}\label{3.9a}
&& \frac{1}{2}e^{-lt}\int_{\Omega}|\partial_t u|^2dx + \frac{1}{2}\frac{d}{dt}(e^{-lt}\int_{\Omega}|\nabla u|^2dx) \nonumber \\
&\leq&\!\!\! C(\Omega)M^2e^{-lt}\int_{\Omega}|\nabla u|^2dx  +(C(\Omega)+\frac{3b^2}{2})M^2e^{-lt}\!\!\int_{\Omega}|u|^2dx -\frac{l}{2}e^{-lt}\int_{\Omega}|\nabla u|^2dx.
 \end{eqnarray}
  Meanwhile, multiplying the first equation of (\ref{1.1}) by $e^{-lt}u$, and integrating it over $\Omega$,
 we see
\begin{eqnarray*}
\int_{\Omega}[\partial_t u-\triangle u+\nabla\cdot(u\nabla v)]\cdot e^{-lt}udx=0,\;t\in[0,T].
 \end{eqnarray*}
 From the above, the second equation of the system \eqref{1.1}, \eqref{4.3bb}, Cauchy's inequality, and Lemma \ref{lemma2.1}, we obtain
 \begin{eqnarray*}
&& \frac{1}{2}\frac{d}{dt}(e^{-lt}\int_{\Omega}|u|^2dx)+e^{-lt}\int_{\Omega}|\nabla u|^2dx  \\
&=& -e^{-lt}\int_{\Omega}(\nabla u \cdot\nabla v+av-bu)udx-\frac{l}{2}e^{-lt}\int_{\Omega}|u|^2dx\\
&\leq& \frac{1}{2}e^{-lt}\int_{\Omega}|\nabla u|^2dx+ C(\Omega)M^2e^{-lt}\int_{\Omega}|u|^2dx+\frac{a}{2}e^{-lt}\int_{\Omega}|v|^2dx\\
 &+&(\frac{a}{2}+b)e^{-lt}\int_{\Omega}|u|^2dx-\frac{l}{2}e^{-lt}\int_{\Omega}|u|^2dx,\:\:t\in[0,T],
 \end{eqnarray*}
from which, it follows that
\begin{eqnarray}\label{3.10a}
&& \frac{d}{dt}(\frac{1}{2}e^{-lt}\int_{\Omega}|u|^2dx)+\frac{1}{2}e^{-lt}\int_{\Omega}|\nabla u|^2dx \nonumber \\
&\leq& (C(\Omega)+a+b)M^2e^{-lt}\int_{\Omega}|u|^2dx -\frac{l}{2}e^{-lt}\int_{\Omega}|u|^2dx,\;t\in[0,T].
 \end{eqnarray}
According to  (\ref{3.9a}), (\ref{3.10a}), and  Poincar\'{e} inequality, there is a constant $l=L_1M^2>0$,
where $L_1>0$ is a constant
 depending only on $\Omega$,  so that
\begin{eqnarray}\label{3.11}
\frac{d}{dt}(e^{-L_1M^2t}\int_{\Omega}|\nabla u(x,t)|^2dx)+e^{-L_1M^2t}\int_{\Omega}|\partial_t u(x,t)|^2dx \leq0,\; t\in[0,T];
 \end{eqnarray}
 \begin{eqnarray}\label{3.12}
 \frac{d}{dt}(e^{-L_1M^2t}\int_{\Omega}|u(x,t)|^2dx)+e^{-L_1M^2t}\int_{\Omega}|\nabla u(x,t)|^2dx\leq0,\; t\in[0,T].
 \end{eqnarray}

 We now show (\ref{3.8aaa}). Integrating (\ref{3.12}) over $(0,t)$, we obtain
 \begin{eqnarray}\label{3.16}
e^{-L_1M^2t}\int_{\Omega}|u(x,t)|^2dx+\int_0^t \big(e^{-L_1M^2s}\int_{\Omega}|\nabla u(x,s)|^2
dx\big)ds\leq \|u_0\|_2^2,\;t\in[0,T],
 \end{eqnarray}
 which leads to (\ref{3.8aaa}).
  By (\ref{3.11}), we  see that the function $t\rightarrow e^{-L_1M^2t}\int_{\Omega}|\nabla u(x,t)|^2dx$   is non-increasing  over $[0, T]$.  Thus,
 \begin{eqnarray}\label{3.15}
te^{-L_1M^2t}\int_{\Omega}|\nabla u(x,t)|^2dx\leq \int_0^t \big(e^{-L_1M^2s}\int_{\Omega}|\nabla u(x,s)|^2dx\big)ds,\;t\in[0,T].
 \end{eqnarray}

Finally, (\ref{3.9aaa}) follows from  (\ref{3.15}) and (\ref{3.16}) at once. This completes the proof.
\end{proof}


\section{Local interpolation inequality for the system (\ref{1.1}) }
\ \ \ \
First of all, since   $\Omega\subset \mathbb{R}^n$ is a bounded domain  with a $C^2$ boundary $\partial\Omega$,
it follows from Theorem 8 in \cite{AEWZ} that for each $g\in\partial\Omega$, there is a triplet $(x_{g}, R_{g}, \delta_g)\in\Omega\times\mathbb{R}^{+}\times(0,1]$  such that
\begin{eqnarray}\label{3.1}
|g-x_{g}|<R_{g}\;  \textrm{ and } \;\Omega\cap B_{(1+2\delta_g)R_{g}} \textrm{ is star-shaped with } x_{g},
\end{eqnarray}
where $B_{(1+2\delta_g)R_{g}}$ stands for the open ball centered at $x_{g}$ and of radius $(1+2\delta_g)R_{g}$.

{\it Throughout this section,   we arbitrarily fix $u_0\in L^\infty(\Omega)$ and $T>0$
so that  the system (\ref{1.1}) has a unique solution $(u,v)$ over $[0,T]$;
 we arbitrarily fix $g\in\partial\Omega$ with $(x_{g}, R_{g}, \delta_g)\in\Omega\times\mathbb{R}^{+}\times(0,1]$ satisfying \eqref{3.1};
we  simply write $B_R$ for the open ball  centered at $x_{g}$ and of radius $R$.}

The main result of this section is the following Theorem \ref{lemma4.4}, which  builds up  a local interpolation inequality for
the solution $(u,v)$.

\begin{theorem}\label{lemma4.4}
For each  $r\in (0,R_g)$ with $B_r:=B(x_g,r)\subset\Omega$,
 there are two constants $D>0$ and $\gamma\in(0,1)$, which depend only on  $\Omega, r, R_g, \delta_g, T$, and $M$ (with $M$ given by \eqref{M}) so that
\begin{eqnarray}\label{4.8aa}
 \int_{\Omega\cap B_{R_g}}|u(x,T)|^2dx
\leq \bigg(D\int_{\Omega}|u_0(x)|^2dx\bigg)^\gamma\bigg(2\int_{B_r}|u(x,T)|^2dx\bigg)^{1-\gamma}.
\end{eqnarray}
\end{theorem}

To prove Theorem \ref{lemma4.4}, we need several lemmas. We start with introducing
two functions in the following manner: Arbitrarily take $z\in H^1(0,T; L^2(\Omega\cap  B_{(1+2\delta_g)R_{g}} ))\cap L^2(0,T; H^2\cap H_0^1(\Omega\cap  B_{(1+2\delta_g)R_{g}} ))$,  then
for each  $\lambda>0$, we define functions
\begin{eqnarray}\label{G}
G_{\lambda}(x,t):=\frac{1}{(T-t+\lambda)^{n/2}}e^{-\frac{|x-x_g|^2}{4(T-t+\lambda)}},\:\: (x,t)\in \mathbb{R}^n\times[0,T]
\end{eqnarray}
and
\begin{eqnarray}\label{N}
N_{\lambda}(t):=\frac{\int_{\Omega\cap  B_{(1+2\delta_g)R_{g}} }|\nabla z(x,t)|^2G_\lambda(x,t)dx}{\int_{\Omega\cap B_{(1+2\delta_g)R_{g}} }|z(x,t)|^2G_\lambda(x,t)dx},
\end{eqnarray}
when $t\in (0,T]$ and $\int_{\Omega\cap  B_{(1+2\delta_g)R_{g}} }|z(x,t)|^2G_\lambda(x,t)dx\neq0$.
\begin{remark}
The above  $N_{\lambda}(\cdot)$
can be viewed as a localized  frequency function.
{\it We simply call it the frequency function.}
\end{remark}

\begin{lemma}\label{lemma4.1}
The frequency function $N_{\lambda}(\cdot)$ (given by \eqref{N}) has the following properties: \\
$(i)$ When $t\in(0,T]$,  $\lambda>0$, and $\int_{\Omega\cap  B_{(1+2\delta_g)R_{g}} }|z(x,t)|^2G_\lambda(x,t)dx\neq0$,
\begin{eqnarray*}\label{4.3}
&&\frac{1}{2}\frac{d}{dt}\int_{\Omega\cap B_{(1+2\delta_g)R_{g}}}\!\!\!|z(x,t)|^2G_\lambda(x,t)dx+N_{\lambda}(t)\int_{\Omega\cap B_{(1+2\delta_g)R_{g}}}\!\!\!|z(x,t)|^2G_\lambda(x,t)dx\nonumber \\
&=&\int_{\Omega\cap B_{(1+2\delta_g)R_{g}}} \!\!\!z(x,t)(\partial_t z(x,t)-\triangle z(x,t))G_{\lambda}(x,t)dx.
\end{eqnarray*}
$(ii)$
  When $t\in(0,T]$,  $\lambda>0$, and $\int_{\Omega\cap  B_{(1+2\delta_g)R_{g}} }|z(x,t)|^2G_\lambda(x,t)dx\neq0$,
\begin{eqnarray*}\label{4.4}
\frac{d}{dt}N_{\lambda} (t)\leq{1\over T-t+\lambda}N_{\lambda}(t)+{\int_{\Omega\cap B_{(1+2\delta_g)R_{g}}} (\partial_t z(x,t)-\triangle z(x,t))^2G_{\lambda}(x,t)dx\over \int_{\Omega\cap B_{(1+2\delta_g)R_{g}}}|z(x,t)|^2G_\lambda(x,t)dx}.
\end{eqnarray*}
\end{lemma}
\begin{proof}
The proof can be done by the same method used in Lemma 3.1 of \cite{Phung1} (see also Lemma 2.3 in \cite{PW}). We omit the details.
\end{proof}

\begin{lemma}\label{theta}
Suppose that $\int_{\Omega\cap B_{R_g}}|u(x,T)|^2dx\neq0$. Then
 there are  positive constants $L_2,L_3, L_4$, and $L_5$ (depending only on $\Omega, R_g$, and $\delta_g$) so that
\begin{eqnarray}\label{3.7,1-25}
0<\frac{\mathbb{E}}{\int_{\Omega\cap B_{(1+\delta_g)R_g}}|u(x,t)|^2dx}\leq e^{\frac{L_2}{\theta}},\:  \textrm{ when}\;\;  T-\theta\leq t\leq T,
\end{eqnarray}
where
\begin{eqnarray}\label{3.8,1-25}
\frac{1}{\theta}:=L_3\ln\bigg( e^{L_4M^2T+L_5(1+\frac{1}{T})}\frac{\mathbb{E}}{\int_{\Omega\cap B_{R_g}}|u(x,T)|^2dx}\bigg)\in \big( \frac{1}{\min\{1,{T}/{2}\}},+\infty\big),
\end{eqnarray}
with $M$ given by (\ref{M}) and
\begin{eqnarray}\label{3.9,1-25}
\mathbb{E}:=\int_{\Omega}|u(x,0)|^2dx+\int_0^T\int_{\Omega}|u(x,s)|^2dxds.
\end{eqnarray}
\end{lemma}
\begin{proof}
Write $R_1:=(1+\delta_g)R_g$.
Let $\sigma\in C_0^\infty(\mathbb{R}^n)$ be such that
\begin{eqnarray}\label{3.8bb}
\textrm{supp}\; \sigma\subset  B_{R_1}, \:   0\leq\sigma\leq1, \textrm{ and } \sigma=1  \textrm{ on }  B_{(1+\delta_g/2)R_g}.
\end{eqnarray}
Then,  there is $C=C(R_g,\delta_g)>0$ so that
\begin{eqnarray}\label{4.4dd}
 |\nabla\sigma(x)|\leq C(R_g,\delta_g), \:  x\in \mathbb{R}^n;
\end{eqnarray}
\begin{eqnarray}\label{4.5aa}
\nabla\sigma(x)=0,\;\;  x\in B_{(1+\delta_g/2)R_g}.
\end{eqnarray}
Multiplying  the first equation of (\ref{1.1}) by $e^{-\frac{|x-x_g|^2}{h}}\sigma^2u$, where $h>0$  will be determined later,  integrating it over $\Omega\cap B_{R_1}$, and then using the integration by parts,
 we find
\begin{eqnarray*}
&&\frac{1}{2}\frac{d}{dt}\int_{\Omega\cap B_{R_1}}|\sigma u|^2\cdot e^{-\frac{|x-x_g|^2}{h}}dx+\int_{\Omega\cap B_{R_1}}\nabla u\cdot\nabla(e^{-\frac{|x-x_g|^2}{h}}\sigma^2u)dx   \nonumber \\
&=&\int_{\Omega\cap B_{R_1}}u\nabla v\cdot\nabla(e^{-\frac{|x-x_g|^2}{h}}\sigma^2u)dx,
\;\;t\in[0,T].
\end{eqnarray*}
Meanwhile, one can directly check
\begin{eqnarray*}
\nabla(e^{-\frac{|x-x_g|^2}{h}}\sigma^2u)=e^{-\frac{|x-x_g|^2}{h}}\sigma^2\nabla u+2e^{-\frac{|x-x_g|^2}{h}}u\sigma\nabla\sigma
-2\sigma^2 ue^{-\frac{|x-x_g|^2}{h}}\frac{x-x_g}{h},\;\; x\in\Omega,\; t\in[0,T].
\end{eqnarray*}
These lead to that when $t\in[0,T]$,
\begin{eqnarray*}
&&\frac{1}{2}\frac{d}{dt}\int_{\Omega\cap B_{R_1}}|\sigma u|^2\cdot e^{-\frac{|x-x_g|^2}{h}}dx+\int_{\Omega\cap B_{R_1}}e^{-\frac{|x-x_g|^2}{h}}\sigma^2|\nabla u|^2dx   \crr\disp
&=&\int_{\Omega\cap B_{R_1}}e^{-\frac{|x-x_g|^2}{h}}\sigma^2u\nabla u\cdot\nabla vdx+\int_{\Omega\cap B_{R_1}}2e^{-\frac{|x-x_g|^2}{h}}\sigma u^2\nabla\sigma\cdot\nabla v dx  \crr\disp
&-& \int_{\Omega\cap B_{R_1}}2e^{-\frac{|x-x_g|^2}{h}}\sigma^2 u^2\cdot\frac{x-x_g}{h}\cdot\nabla vdx-\int_{\Omega\cap B_{R_1}}2e^{-\frac{|x-x_g|^2}{h}}u\sigma\nabla \sigma\cdot\nabla udx \crr\disp
&+&\int_{\Omega\cap B_{R_1}}2e^{-\frac{|x-x_g|^2}{h}}\sigma^2 u\frac{x-x_g}{h}\cdot\nabla udx.
\end{eqnarray*}
 The above equation,  along with  \eqref{4.3bb}, yields that when $t\in[0,T]$,
\begin{eqnarray*}
&&\frac{1}{2}\frac{d}{dt}\int_{\Omega\cap B_{R_1}}|\sigma u|^2\cdot e^{-\frac{|x-x_g|^2}{h}}dx+\int_{\Omega\cap B_{R_1}}e^{-\frac{|x-x_g|^2}{h}}\sigma^2|\nabla u|^2dx   \nonumber \\
&\leq&\int_{\Omega\cap B_{R_1}}C_1Me^{-\frac{|x-x_g|^2}{h}}\sigma^2|u|\cdot|\nabla u|dx+C_1M\int_{\Omega\cap B_{R_1}}|\sigma||\nabla\sigma|e^{-\frac{|x-x_g|^2}{h}} |u|^2dx  \nonumber \\
&+& C_1M\cdot\frac{R_1}{h}\int_{\Omega\cap B_{R_1}}e^{-\frac{|x-x_g|^2}{h}}\sigma^2 |u|^2dx+\int_{\Omega\cap B_{R_1}}2|\sigma||\nabla\sigma|e^{-\frac{|x-x_g|^2}{h}}|u|\cdot|\nabla u|dx \nonumber \\
&+&\int_{\Omega\cap B_{R_1}}2\frac{R_1}{h}e^{-\frac{|x-x_g|^2}{h}}\sigma^2 |u|\cdot|\nabla u|dx,
\end{eqnarray*}
  where $C_1$ is a positive constant depending only on $\Omega$. Using Cauchy's inequality in the above inequality
  gives
\begin{eqnarray*}
&&\frac{1}{2}\frac{d}{dt}\int_{\Omega\cap B_{R_1}}|\sigma u|^2\cdot e^{-\frac{|x-x_g|^2}{h}}dx+\int_{\Omega\cap B_{R_1}}e^{-\frac{|x-x_g|^2}{h}}|\sigma\nabla u|^2dx   \nonumber \\
&\leq&\frac{1}{2}\int_{\Omega\cap B_{R_1}}e^{-\frac{|x-x_g|^2}{h}}|\sigma\nabla u|^2dx+8\int_{\Omega\cap B_{R_1}}e^{-\frac{|x-x_g|^2}{h}}|\nabla\sigma|^2|u|^2dx \nonumber\\
&+&C_1(M^2+\frac{R_1^2}{h^2})\int_{\Omega\cap B_{R_1}}|\sigma u|^2\cdot e^{-\frac{|x-x_g|^2}{h}}dx
,\;\;
t\in[0,T].
\end{eqnarray*}
Moving the term $\frac{1}{2}\int_{\Omega\cap B_{R_1}}e^{-\frac{|x-x_g|^2}{h}}|\sigma\nabla u|^2dx$  to the left hand side in the above, using  \eqref{4.4dd} and \eqref{4.5aa}, we  deduce
that when $t\in[0,T]$,
\begin{eqnarray*}
&&\frac{d}{dt}\int_{\Omega\cap B_{R_1}}|\sigma u|^2\cdot e^{-\frac{|x-x_g|^2}{h}}dx\leq
C_2\int_{\Omega\cap (B_{R_1}\setminus B_{(1+\delta_g/2)R_g})}e^{-\frac{|x-x_g|^2}{h}}|u|^2dx  \nonumber \\
&+&C_1(M^2+\frac{R_1^2}{h^2})\int_{\Omega\cap B_{R_1}}|\sigma u|^2\cdot e^{-\frac{|x-x_g|^2}{h}}dx, \nonumber \\
&\leq& C_2e^{-\frac{(1+\delta_g/2)^2R_g^2}{h}}\int_{\Omega\cap B_{R_1}}|u|^2dx
+C_1(M^2+\frac{R_1^2}{h^2})\int_{\Omega\cap B_{R_1}}|\sigma u|^2\cdot e^{-\frac{|x-x_g|^2}{h}}dx,
\end{eqnarray*}
where $C_2>1$ is a constant  depending only  on $R_g$ and $\delta_g$.
 Multiplying the above inequality by $\exp\big(-C_1(M^2+\frac{R_1^2}{h^2})t\big)$ and then integrating it between
  $t$ and $T$,  we obtain
\begin{eqnarray*}
&&\int_{\Omega\cap B_{R_1}}|\sigma u(x,T)|^2\cdot e^{-\frac{|x-x_g|^2}{h}}dx \nonumber \\
&\leq&\exp\big(C_1(M^2+\frac{R_1^2}{h^2})(T-t)\big)\int_{\Omega\cap B_{R_1}}|\sigma u(x,t)|^2\cdot e^{-\frac{|x-x_g|^2}{h}}dx  \nonumber \\
&+& C_2e^{-\frac{(1+\delta_g/2)^2R_g^2}{h}}\exp\big(C_1(M^2+\frac{R_1^2}{h^2})(T-t)\big)\int_t^T\int_{\Omega\cap B_{R_1}}|u(x,s)|^2dxds,\;\;t\in[0,T],
\end{eqnarray*}
from which and \eqref{3.8bb}, it follows that when $t\in[0,T]$,
\begin{eqnarray}\label{4.4cc}
&&\int_{\Omega\cap B_{R_g}}|u(x,T)|^2dx \leq e^{\frac{R_g^2}{h}}\int_{\Omega\cap B_{R_g}}|\sigma u(x,T)|^2\cdot e^{-\frac{|x-x_g|^2}{h}}dx \crr\disp
&\leq &e^{\frac{R_g^2}{h}}\exp\big(C_1(M^2+\frac{R_1^2}{h^2})(T-t)\big)\int_{\Omega\cap B_{R_1}}|u(x,t)|^2dx  \crr\disp
&+& C_2e^{\frac{-(1+\delta_g/2)^2R_g^2+R_g^2}{h}}\exp\big(C_1(M^2+\frac{R_1^2}{h^2})(T-t)\big)\int_t^T\int_{\Omega\cap B_{R_1}}|u(x,s)|^2dxds.
\end{eqnarray}

Next, we let
\begin{eqnarray}\label{4.4bb}
l:=\frac{\delta_g+\delta_g^2/4}{2C_1(1+\delta_g)^2}.
\end{eqnarray}
Then we have
\begin{eqnarray}\label{4.5}
\frac{C_1lR_1^2}{h}={\frac{(1+\delta_g/2)^2R_g^2-R_g^2}{2h}}.
\end{eqnarray}
Choosing $h$ sufficiently small so that $0<lh<\min\{1,\frac{T}{2}\}$, and then using \eqref{4.4cc}, we see that  when $\frac{T}{2}<T-lh\leq t\leq T$,
\begin{eqnarray*}
&&\int_{\Omega\cap B_{R_g}}|u(x,T)|^2dx \leq e^{\frac{R_g^2}{h}}\exp\big(C_1(M^2+\frac{R_1^2}{h^2})lh\big)\int_{\Omega\cap B_{R_1}}|u(x,t)|^2dx  \nonumber \\
&+& C_2e^{\frac{-(1+\delta_g/2)^2R_g^2+R_g^2}{h}}\exp\big(C_1(M^2+\frac{R_1^2}{h^2})lh\big)\int_t^T\int_{\Omega\cap B_{R_1}}|u(x,s)|^2dxds.
\end{eqnarray*}
This, together with (\ref{4.5}) and \eqref{3.9,1-25}, shows that  when $\frac{T}{2}<T-lh\leq t\leq T$,
\begin{eqnarray}\label{3.12aa}
\int_{\Omega\cap B_{R_g}}|u(x,T)|^2dx &\leq &e^{{\frac{(1+\delta_g/2)^2R_g^2+R_g^2}{2h}}}\exp\big(C_1M^2lh\big)\int_{\Omega\cap B_{R_1}}|u(x,t)|^2dx  \nonumber \\
&+& C_2e^{\frac{-(1+\delta_g/2)^2R_g^2+R_g^2}{2h}}\exp\big(C_1M^2lh\big)\mathbb{E}.
\end{eqnarray}
Meanwhile, since $\int_{\Omega\cap B_{R_g}}|u(x,T)|^2dx\neq0$,  we have $u(\cdot,T)\neq0$.
This, along with  the continuity  of $u$ and \eqref{3.9,1-25}, shows that
\begin{eqnarray}\label{3.16,1-25}
\mathbb{E}>0.
\end{eqnarray}

Now, we take
\begin{eqnarray}\label{4.4aa}
h:=\frac{R_g^2(\delta_g+\delta_g^2/4)/2}{\ln\bigg( \mathcal{K}\frac{C_2\exp\big((C_1+L_1)M^2T\big)\mathbb{E}}{\frac{1}{e}\int_{\Omega\cap B_{R_g}}|u(x,T)|^2dx}\bigg)},
\end{eqnarray}
where $\mathcal{K}=e^{(R_g^2(\delta_g+\delta_g^2/4)/2)(\frac{2}{T}+1) l}$ and $L_1$ is given by Theorem \ref{proposition3.2}.
  One can directly check from \eqref{4.4aa}, \eqref{3.9,1-25}, \eqref{3.8aaa}, and $C_2>1$  that $0<lh<\min\{1,\frac{T}{2}\}$ and
 \begin{eqnarray}\label{3.18,1-25}
 C_2e^{\frac{-(1+\delta_g/2)^2R_g^2+R_g^2}{2h}}\exp\big(C_1M^2T\big)\mathbb{E}\leq\frac{1}{e}\int_{\Omega\cap B_{R_g}}|u(x,T)|^2dx.
\end{eqnarray}
Then by  \eqref{3.18,1-25}, and \eqref{3.12aa},
we find that
when
 $\frac{T}{2}<T-lh\leq t\leq T$,
\begin{eqnarray*}
(1-\frac{1}{e})\int_{\Omega\cap B_{R_g}}|u(x,T)|^2dx\leq e^{{\frac{(1+\delta_g/2)^2R_g^2+R_g^2}{2h}}}\exp\big(C_1M^2T\big)\int_{\Omega\cap B_{R_1}}|u(x,t)|^2dx,
\end{eqnarray*}
which, along with \eqref{3.16,1-25},  \eqref{3.18,1-25}, and $C_2>1$, leads to
\begin{eqnarray}\label{4.3aa}
0<\mathbb{E}\leq e^{\frac{(1+\delta_g/2)^2R_g^2}{h}}\int_{\Omega\cap B_{R_1}}|u(x,t)|^2dx,\;\;
\frac{T}{2}<T-lh\leq t\leq T.
\end{eqnarray}
Let  $\theta:=lh$. Then it follows from (\ref{4.4bb}), (\ref{4.4aa}), and (\ref{4.3aa})
 that when $\frac{T}{2}<T-\theta\leq t\leq T$,
\begin{eqnarray*}
0<\frac{\mathbb{E}}{\int_{\Omega\cap B_{R_1}}|u(x,t)|^2dx}\leq e^{\frac{(1+\delta_g/2)^2R_g^2}{h}}= e^{\frac{(\delta_g+\delta_g^2/4)(1+\delta_g/2)^2R_g^2}{2C_1(1+\delta_g)^2}\frac{1}{\theta}};
\end{eqnarray*}
\begin{eqnarray*}
\frac{1}{\theta}=\frac{4C_1(1+\delta_g)^2}{R_g^2(\delta_g+\delta_g^2/4)^2}\ln\big( eC_2\exp\big((C_1+L_1)M^2T\big)e^{(\frac{2}{T}+1)\frac{R_g^2(\delta_g+\delta_g^2/2)^2}{4C_1(1+\delta_g)^2}}\frac{\mathbb{E}}{\int_{\Omega\cap B_{R_g}}|u(x,T)|^2dx}\bigg).
\end{eqnarray*}
These lead to \eqref{3.7,1-25} and \eqref{3.8,1-25}.
Thus we  complete the proof.
\end{proof}
\begin{remark}\label{remark4.1}
 Lemma \ref{theta} implies that if $\int_{\Omega\cap B_{R_g}}|u(x,T)|^2dx\neq0$, then $\int_{\Omega\cap B_{(1+\delta_g)R_g}}|u(x,t)|^2dx\neq0$, for any $T-\theta\leq t\leq T$, where $\theta$ is  given by \eqref{3.8,1-25}.
\end{remark}

 The following lemma is quoted from \cite{EscauriazaF1, PW} (see, for instance, (2.3.15) on page 691 in \cite{PW}).
\begin{lemma}\label{lemma 4.3}
Let $G$ be a bounded domain in $\mathbb{R}^n$ with a $C^2$ boundary $\partial G$. Let $x_0\in G$.
Then, for each $f \in H_0^1(G)$  and each  $\lambda>0$,
\begin{eqnarray*}\label{4.28}
\int_G\frac{|x-x_0|^2}{8\lambda}|f(x)|^2e^{-\frac{|x-x_0|^2}{4\lambda}}dx\leq2\lambda\int_G|\nabla f(x)|^2e^{-\frac{|x-x_0|^2}{4\lambda}}dx+\frac{n}{2}\int_G|f(x)|^2e^{-\frac{|x-x_0|^2}{4\lambda}}dx.
\end{eqnarray*}
\end{lemma}

Now, we are in the position to prove
Theorem \ref{lemma4.4}.

\begin{proof}[Proof of Theorem \ref{lemma4.4}]
Arbitrarily fix $r\in (0,R_g)$ with $B_r:=B(x_g,r)\subset\Omega$.
Without loss of generality, we can assume  that $\int_{\Omega\cap B_{R_g}}|u(x,T)|^2dx\neq0$.
Let $\theta$ be given  in Lemma \ref{theta} and $R_0:=(1+2\delta_g)R_g$.
Let $\sigma_0\in C_0^\infty(\mathbb{R}^n)$ satisfy that
\begin{eqnarray}\label{truncation}
\textrm{supp}\; \sigma_0\subset B_{R_0},\; 0\leq\sigma_0\leq1,\;\;\mbox{and}\;\;\sigma_0=1  \;\;\mbox{on}\;\;  B_{(1+3\delta_g/2)R_g}.
\end{eqnarray}
Write  $\phi:=\sigma_0\cdot u$.
 We organize the rest of the  proof in several steps.\\

\noindent{\it Step 1.  We present several observations on the truncated function  $\phi$.}

Observation One: By direct  computations, we find
\begin{eqnarray}\label{cut}
\phi_t-\triangle\phi=-\nabla\phi\cdot\nabla v-\phi\triangle v+u\nabla v\cdot\nabla\sigma_0-2\nabla u\cdot\nabla\sigma_0-u\triangle\sigma_0,\;\;x\in\Omega,\; t\in[0,T].
\end{eqnarray}

Observation Two: If we set
\begin{eqnarray}\label{3.24,1.30}
\psi:=u\nabla v\cdot\nabla\sigma_0-2\nabla u\cdot\nabla\sigma_0-u\triangle\sigma_0,\;\;x\in\Omega,\; t\in[0,T],
\end{eqnarray}
then we have
\begin{eqnarray}\label{4.12}
\psi(x,t)=0,\: \;\mbox{when}\;\; x\in B_{(1+3\delta_g/2)R_g}, \; t\in[0, T];
\end{eqnarray}
\begin{eqnarray}\label{4.13aa}
&&\int_{\Omega\cap B_{R_0}} |\psi(x,t)|^2dx\leq C_1M^2(1+t^{-1})e^{L_1M^2t}\int_{\Omega} |u(x,0)|^2dx,
\; t\in(0,T],
\end{eqnarray}
where $L_1$ is given in Theorem \ref{proposition3.2} and where $C_1>0$ is a constant (depending only on $\Omega, R_g$, and $\delta_g$).
Indeed,  \eqref{4.12} follows directly from \eqref{3.24,1.30} and \eqref{truncation}.
To show \eqref{4.13aa}, we first notice that by \eqref{truncation},  there is $C:=C(R_g,\delta_g)$ so that
\begin{eqnarray}\label{sigma}
|\triangle\sigma_0(x)|\leq C(R_g,\delta_g)\;\;\mbox{and}\;\;  |\nabla\sigma_0(x)|\leq C(R_g,\delta_g),\;\;x\in\mathbb{R}^n;
\end{eqnarray}
 we then use \eqref{3.24,1.30} and Cauchy's inequality  to get that when $t\in [0,T]$,
\begin{eqnarray*}
&&\int_{\Omega\cap B_{R_0}} |\psi(x,t)|^2dx=\int_{\Omega\cap B_{R_0}} \big|\big[u\nabla v\cdot\nabla\sigma_0-2\nabla u\cdot\nabla\sigma_0-u\triangle\sigma_0\big](x,t)\big|^2dx\crr\disp
&\leq&\int_{\Omega\cap B_{R_0}} \big[(|u|^2+|\nabla u|^2+|u|^2)(x,t)\big]\big[(|\nabla v\cdot\nabla\sigma_0|^2+4|\nabla\sigma_0|^2+|\triangle\sigma_0|^2)(x,t)\big]dx;
\end{eqnarray*}
we finally use the above inequality, \eqref{4.3bb}, \eqref{sigma} and  Theorem \ref{proposition3.2}
to get \eqref{4.13aa}.

Observation Three: Taking $z=\phi$ in \eqref{N}, we see from Lemma \ref{lemma4.1} that
\begin{eqnarray}\label{4.8}
&&\frac{1}{2}\frac{d}{dt}\int_{\Omega\cap B_{R_0}}|\phi(x,t)|^2G_\lambda(x,t)dx+N_{\lambda}(t)\int_{\Omega\cap B_{R_0}}|\phi(x,t)|^2G_\lambda(x,t)dx\nonumber \\
&=&\int_{\Omega\cap B_{R_0}} \phi(\partial_t \phi(x,t)-\triangle \phi(x,t))G_{\lambda}(x,t)dx,\; \lambda>0,\; t\in[T-\theta,T]
\end{eqnarray}
and
\begin{eqnarray}\label{4.9}
\frac{d}{dt}N_{\lambda} (t)\leq{1\over T-t+\lambda}N_{\lambda}(t)+{\int_{\Omega\cap B_{R_0}} (\partial_t \phi(x,t)-\triangle \phi(x,t))^2G_{\lambda}dx\over \int_{\Omega\cap B_{R_0}}|\phi(x,t)|^2G_\lambda(x,t)dx},
\; \lambda>0,\; t\in[T-\theta,T].
\end{eqnarray}
To show these, we arbitrarily fix $\lambda>0$. We claim
\begin{eqnarray}\label{3.29,1-30}
\int_{\Omega\cap B_{R_0}}|\phi(x,t)|^2G_\lambda(x,t)dx\neq0,\;\;t\in[T-\theta,T].
\end{eqnarray}
Indeed, by \eqref{truncation}, the fact that  $\phi:=\sigma_0\cdot u$, and $R_0:=(1+2\delta_g)R_g$,
 we can easily see
\begin{eqnarray}\label{3.25}
\int_{\Omega\cap B_{R_0}}|\phi(x,t)|^2G_\lambda(x,t)dx\geq\int_{\Omega\cap B_{(1+\delta_g)R_g}}|u(x,t)|^2G_\lambda(x,t)dx,\;\;t\in[0,T].
\end{eqnarray}
Since $\int_{\Omega\cap B_{R_g}}|u(x,T)|^2dx\neq0$, it follows from
   Remark \ref{remark4.1} and
   \eqref{G}
   that when $t\in[T-\theta, T]$,
\begin{eqnarray*}
\int_{\Omega\cap B_{(1+\delta_g)R_g}}|u(x,t)|^2G_\lambda(x,t)dx\geq\frac{1}{(T-t+\lambda)^{n/2}}e^{-\frac{(1+\delta_g)^2R_g^2}{4(T-t+\lambda)}}\int_{\Omega\cap B_{(1+\delta_g)R_g}}|u(x,t)|^2dx>0,
\end{eqnarray*}
 which, together with \eqref{3.25}, leads to
 \eqref{3.29,1-30}.
 Next,  by \eqref{3.29,1-30}, we can use Lemma \ref{lemma4.1}, where
 $z$ is replaced by $\phi$, to get \eqref{4.8} and \eqref{4.9}.

\noindent{\it Step 2.  We show that for each $\lambda>0$ and $\varepsilon\in(0,\theta)$,
\begin{eqnarray}\label{estimate3}
\frac{\int_{\Omega\cap B_{R_0}} |\psi(x,t)|^2G_{\lambda}(x,t)dx}{\int_{\Omega\cap B_{R_0}}|\phi(x,t)|^2G_\lambda(x,t)dx}\leq 2C_1M^2(1+\frac{1}{T})e^{L_1M^2T}e^{\frac{L_2}{\theta}}e^{-\frac{L_6}{\varepsilon+\lambda}},
\;\;t\in[T-\varepsilon,T],
\end{eqnarray}
where
\begin{eqnarray}\label{3.32;1.30}
L_6:=-\frac{(1+\delta_g)^2R_g^2}{4}+\frac{(1+3\delta_g/2)^2R_g^2}{4},
\end{eqnarray}
and where $L_2$ is given in Theorem \ref{proposition3.2}.}

To this end, we arbitrarily fix $\lambda>0$ and $\varepsilon\in(0,\theta)$.
By  \eqref{3.8,1-25}, we have  $0<\varepsilon<\theta<\min\{1,\frac{T}{2}\}$,
from which, it follows that
 \begin{eqnarray}\label{3.27}
t^{-1}\leq\frac{2}{T},  \: \textrm{ when } t\in[T-\varepsilon,T].
\end{eqnarray}
Meanwhile, by   \eqref{4.12} and \eqref{3.25},  we see that when $t\in[T-\varepsilon,T]$,
\begin{eqnarray*}
\frac{\int_{\Omega\cap B_{R_0}} |\psi(x,t)|^2G_{\lambda}(x,t)dx}{\int_{\Omega\cap B_{R_0}}|\phi(x,t)|^2G_\lambda(x,t)dx}
\leq
 \frac{\int_{\Omega\cap (B_{R_0}\setminus B_{(1+3\delta_g/2)R_g})} |\psi(x,t)|^2G_{\lambda}(x,t)dx}{\int_{\Omega\cap  B_{(1+\delta_g)R_g}}|u(x,t)|^2G_{\lambda}(x,t)dx},
\end{eqnarray*}
which, along with \eqref{G} and \eqref{3.32;1.30}, yields  that when $t\in[T-\varepsilon,T]$,
\begin{eqnarray*}
&&\frac{\int_{\Omega\cap B_{R_0}} |\psi(x,t)|^2G_{\lambda}(x,t)dx}{\int_{\Omega\cap B_{R_0}}|\phi(x,t)|^2G_\lambda(x,t)dx}\leq \frac{\int_{\Omega\cap (B_{R_0}\setminus B_{(1+3\delta_g/2)R_g})} |\psi(x,t)|^2dx}{\int_{\Omega\cap  B_{(1+\delta_g)R_g}}|u(x,t)|^2dx}e^{-\frac{L_6}{T-t+\lambda}}.
\end{eqnarray*}
This, together with  \eqref{4.13aa}, shows that when $t\in[T-\varepsilon,T]$,
\begin{eqnarray}\label{3.33,1-20}
\frac{\int_{\Omega\cap B_{R_0}} |\psi(x,t)|^2G_{\lambda}(x,t)dx}{\int_{\Omega\cap B_{R_0}}|\phi(x,t)|^2G_\lambda(x,t)dx}
\leq \frac{C_1M^2(1+t^{-1})e^{L_1M^2t}\int_{\Omega} |u(x,0)|^2dx}{\int_{\Omega\cap B_{(1+\delta_g)R_g}}|u(x,t)|^2dx}e^{-\frac{L_6}{\varepsilon+\lambda}}.
\end{eqnarray}
Now, by \eqref{3.33,1-20}, Lemma \ref{theta}, and \eqref{3.27},
we see that \eqref{estimate3} holds for all $t\in[T-\varepsilon,T]$.

\noindent{\it Step 3. We
show that for any $\lambda>0$ and $\varepsilon\in(0,\theta)$,
\begin{eqnarray}\label{4.24aa}
&&\lambda N_{\lambda}(T)\leq\frac{2(\lambda+\varepsilon)}{\varepsilon}\cdot\exp(C_2M^2\varepsilon) \crr\disp
&\times&\big[\frac{C_3}{2}M^2\varepsilon+\frac{C_2}{2}M^2\varepsilon^2
+\frac{\varepsilon}{2}Q(\theta,\varepsilon, \lambda)+\frac{(1+\delta_g)^2R_g^2}{2\varepsilon}+L_1M^2T+\frac{L_2}{\theta}\big],
\end{eqnarray}
where   $C_2$, $C_3$ are positive constants depending only on $\Omega$, and where
\begin{eqnarray}\label{3.34}
Q(\theta,\varepsilon, \lambda):=6C_1M^2(1+\frac{1}{T})e^{L_1M^2T}e^{\frac{L_2}{\theta}}
 e^{-\frac{L_6}{\varepsilon+\lambda}}(1+\varepsilon).
\end{eqnarray}
}

Arbitrarily fix $\lambda>0$ and $\varepsilon\in(0,\theta)$.
It follows from (\ref{cut}), (\ref{4.9}), and the second equation of (\ref{1.1}) that when $t\in[T-\varepsilon,T]$,
\begin{eqnarray*}
\frac{d}{dt}N_{\lambda}(t)\leq
{1\over T-t+\lambda}N_{\lambda}(t)+{\int_{\Omega\cap B_{R_0}} \big[(-\nabla\phi\cdot\nabla v-\phi(av-bu)+\psi)(x,t)\big]^2G_{\lambda}(x,t)dx\over \int_{\Omega\cap B_{R_0}}|\phi(x,t)|^2G_\lambda(x,t) dx}.
\end{eqnarray*}
This, along with Cauchy's inequality,  \eqref{M},  and (\ref{4.3bb}), yields that when $t\in[T-\varepsilon,T]$,
\begin{eqnarray*}
\frac{d}{dt}N_{\lambda}(t)\!\!&\leq&\!\!
{1\over T-t+\lambda}N_{\lambda}(t)+{\int_{\Omega\cap B_{R_0}} 3\big(|\nabla\phi|^2|\nabla v|^2+|\phi|^2|(av-bu)|^2+|\psi|^2\big)(x,t) G_{\lambda}(x,t)dx\over \int_{\Omega\cap B_{R_0}}|\phi(x,t)|^2G_\lambda(x,t)dx}\crr\disp
&\leq&({1\over T-t+\lambda}+C_2M^2)N_{\lambda}(t)+C_2M^2+3{\int_{\Omega\cap B_{R_0}} |\psi(x,t)|^2G_{\lambda}(x,t)dx\over \int_{\Omega\cap B_{R_0}}|\phi(x,t)|^2G_\lambda(x,t)dx},
\end{eqnarray*}
for some  $C_2>0$  depending only on $\Omega$.
It, together with \eqref{estimate3}, implies that when $t\in[T-\varepsilon,T]$,
\begin{eqnarray*}
\frac{d}{dt}N_{\lambda} (t)-({1\over T-t+\lambda}+C_2M^2)N_{\lambda}(t)\leq C_2M^2+6C_1M^2(1+\frac{1}{T})e^{L_1M^2T}e^{\frac{L_2}{\theta}}e^{-\frac{L_6}{\varepsilon+\lambda}}.
\end{eqnarray*}
Multiplying the above  by $\exp(\ln(T -t+\lambda)-C_2M^2t)$, and then integrating it over $(t, T)$, we obtain
\begin{eqnarray*}
&&\lambda \exp(-C_2M^2T)N_{\lambda}(T)\leq (T-t+\lambda)\exp(-C_2M^2t)N_{\lambda}(t)  \nonumber \\
&+&\big(C_2M^2+6C_1M^2(1+\frac{1}{T})e^{L_1M^2T}e^{\frac{L_2}{\theta}}e^{-\frac{L_6}{\varepsilon+\lambda}}\big)
\int_{t}^{T}e^{-C_2M^2s}(T-s+\lambda)ds,\;\; t\in[T-\varepsilon,T].
\end{eqnarray*}
Dividing both sides of the above by $\exp(-C_2M^2T)$, we see that when $t\in[T-\varepsilon,T]$,
\begin{eqnarray*}
&&\lambda N_{\lambda}(T)\leq (T-t+\lambda)\exp\big(C_2M^2(T-t)\big)N_{\lambda}(t)  \nonumber \\
&+&\big(C_2M^2+6C_1M^2(1+\frac{1}{T})e^{L_1M^2T}e^{\frac{L_2}{\theta}}e^{-\frac{L_6}{\varepsilon+\lambda}}\big)
\int_{t}^{T}e^{C_2M^2(T-s)}(T-s+\lambda)ds \nonumber \\
&\leq& (\lambda+\varepsilon)\exp\big(C_2M^2\varepsilon\big)\big[N_{\lambda}(t)+
\big(C_2M^2+6C_1M^2(1+\frac{1}{T})e^{L_1M^2T}e^{\frac{L_2}{\theta}}e^{-\frac{L_6}{\varepsilon+\lambda}}\big)\varepsilon\big].
\end{eqnarray*}
 This implies that  when $t\in[T-\varepsilon,T]$,
\begin{eqnarray}\label{estimate1}
\frac{\lambda}{\lambda+\varepsilon}\exp(-C_2M^2\varepsilon) N_{\lambda}(T)-\big(C_2M^2+6C_1M^2(1+\frac{1}{T})e^{L_1M^2T}e^{\frac{L_2}{\theta}}e^{-\frac{L_6}{\varepsilon+\lambda}}\big)\varepsilon\leq N_{\lambda}(t).
\end{eqnarray}
Meanwhile, it follows from (\ref{cut}) and (\ref{4.8}) that
\begin{eqnarray*}
&&\frac{1}{2}\frac{d}{dt}\int_{\Omega\cap B_{R_0}}|\phi(x,t)|^2G_\lambda(x,t)dx+N_{\lambda}(t)\int_{\Omega\cap B_{R_0}}|\phi(x,t)|^2G_\lambda(x,t) dx\crr\disp
&=&\int_{\Omega\cap B_{R_0}} \phi(x,t)\big[-\nabla\phi\cdot\nabla v-\phi\triangle v+\psi\big](x,t)G_{\lambda}(x,t)dx,\;\;t\in[T-\varepsilon,T].
\end{eqnarray*}
This, along with  the second equation of (\ref{1.1}),  (\ref{4.3bb}), \eqref{M}, and Cauchy's inequality, indicates that when $t\in[T-\varepsilon,T]$,
\begin{eqnarray*}
&&\frac{1}{2}\frac{d}{dt}\int_{\Omega\cap B_{R_0}}|\phi(x,t)|^2G_\lambda(x,t)dx+N_{\lambda}(t)\int_{\Omega\cap B_{R_0}}|\phi(x,t)|^2G_\lambda(x,t)dx\crr\disp
&\leq&\frac{1}{2}\int_{\Omega\cap B_{R_0}}|\nabla\phi(x,t)|^2G_{\lambda}(x,t)dx+\frac{1}{2}\int_{\Omega\cap B_{R_0}}|\phi(x,t)\nabla v(x,t)|^2G_{\lambda}(x,t)dx\crr\disp
&+&\int_{\Omega\cap B_{R_0}} \big[\phi^2(bu-av)+\phi\psi\big](x,t)G_{\lambda}(x,t)dx\crr\disp
&\leq&\bigg(\frac{1}{2}N_{\lambda}(t)+C_3M^2\bigg)\int_{\Omega\cap B_{R_0}}|\phi(x,t)|^2G_\lambda(x,t)dx+\int_{\Omega\cap B_{R_0}} |\psi(x,t)|^2G_{\lambda}(x,t)dx,
\end{eqnarray*}
where  $C_3$ is a positive constant  depending only on $\Omega$.
Then, after some computations, we obtain that when $t\in[T-\varepsilon,T]$,
\begin{eqnarray}\label{3.38,,1.30}
&&\frac{1}{2}\frac{d}{dt}\int_{\Omega\cap B_{R_0}}|\phi(x,t)|^2G_\lambda(x,t)dx+\frac{1}{2}N_{\lambda}(t)\int_{\Omega\cap B_{R_0}}|\phi(x,t)|^2G_\lambda(x,t)dx\nonumber \\
&\leq&\big(C_3M^2+\frac{\int_{\Omega\cap B_{R_0}} |\psi(x,t)|^2G_{\lambda}(x,t)dx}{\int_{\Omega\cap B_{R_0}}|\phi(x,t)|^2G_\lambda(x,t)dx}\big)\int_{\Omega\cap B_{R_0}}|\phi(x,t)|^2G_\lambda(x,t)dx.
\end{eqnarray}
Thus,  by \eqref{3.38,,1.30},  (\ref{estimate1}), and  (\ref{estimate3}),
we find that  when $t\in[T-\varepsilon,T]$,
\begin{eqnarray*}
&&\frac{d}{dt}\int_{\Omega\cap B_{R_0}}|\phi(x,t)|^2G_\lambda(x,t)dx+\frac{\lambda}{\lambda+\varepsilon}\exp(-C_2M^2\varepsilon) N_{\lambda}(T)\int_{\Omega\cap B_{R_0}}|\phi(x,t)|^2G_\lambda(x,t)dx\nonumber \\
&\leq&\!\!\!\!\big(C_3M^2\!+\!C_2M^2\varepsilon
+6C_1M^2(1+\frac{1}{T})e^{L_1M^2T}e^{\frac{L_2}{\theta}}e^{-\frac{L_6}{\varepsilon+\lambda}}(1+\varepsilon)\big)\!\!\!\int_{\Omega\cap B_{R_0}}\!\!\!|\phi(x,t)|^2G_\lambda(x,t)dx.
\end{eqnarray*}
 This, along with \eqref{3.34}, yields that  when $t\in[T-\varepsilon,T]$,
\begin{eqnarray*}
\frac{d}{dt}\bigg(e^{(\frac{\lambda}{\lambda+\varepsilon}\exp(-C_2M^2\varepsilon) N_{\lambda}(T)-C_3M^2-C_2M^2\varepsilon
-Q(\theta,\varepsilon, \lambda))t}\int_{\Omega\cap B_{R_0}}|\phi(x,t)|^2G_\lambda(x,t)dx\bigg)\leq0.
\end{eqnarray*}
Integrating it over $(T-\varepsilon, T-\frac{\varepsilon}{2})$, we have
\begin{eqnarray*}
&&e^{\frac{\lambda}{\lambda+\varepsilon}\exp(-C_2M^2\varepsilon) N_{\lambda}(T)\frac{\varepsilon}{2}}\int_{\Omega\cap  B_{R_0}}|\phi(x,T-\frac{\varepsilon}{2})|^2G_\lambda(x,T-\frac{\varepsilon}{2})dx\\
&\leq& e^{(C_3M^2+C_2M^2\varepsilon
+Q(\theta,\varepsilon, \lambda))\frac{\varepsilon}{2}}\int_{\Omega\cap B_{R_0}}|\phi(x,T-\varepsilon)|^2G_\lambda(x,T-\varepsilon)dx,
\end{eqnarray*}
from which, it follows that
\begin{eqnarray}\label{4.23}
&&e^{\frac{\varepsilon}{2(\lambda+\varepsilon)}\exp(-C_2M^2\varepsilon) \lambda N_{\lambda}(T)}\nonumber \\
&\leq& e^{(C_3M^2+C_2M^2\varepsilon
+Q(\theta,\varepsilon, \lambda))\frac{\varepsilon}{2}}\frac{\int_{\Omega\cap B_{R_0}}|\phi(x,T-\varepsilon)|^2G_\lambda(x,T-\varepsilon)dx}{\int_{\Omega\cap  B_{R_0}}|\phi(x,T-\frac{\varepsilon}{2})|^2G_\lambda(x,T-\frac{\varepsilon}{2})dx}.
\end{eqnarray}

We next estimate the right hand side of \eqref{4.23}.
One can directly check
\begin{eqnarray}\label{3.32}
\frac{\int_{\Omega\cap B_{R_0}}|\phi(x,T-\varepsilon)|^2G_\lambda(x,T-\varepsilon)dx}{\int_{\Omega\cap  B_{R_0}}|\phi(x,T-\frac{\varepsilon}{2})|^2G_\lambda(x,T-\frac{\varepsilon}{2})dx}&\leq&\frac{\int_{\Omega\cap B_{R_0}}|\phi(x,T-\varepsilon)|^2e^{-\frac{|x-x_g|^2}{4(\varepsilon+\lambda)}}dx}{\int_{\Omega\cap  B_{R_0}}|\phi(x,T-\frac{\varepsilon}{2})|^2e^{-\frac{|x-x_g|^2}{4(\varepsilon/2+\lambda)}}dx}\crr\disp
&\leq&\frac{\int_{\Omega\cap B_{R_0}}|\phi(x,T-\varepsilon)|^2dx}{ e^{-\frac{(1+\delta_g)^2R_g^2}{2\varepsilon}}\int_{\Omega\cap  B_{(1+\delta_g)R_g}}|\phi(x,T-\frac{\varepsilon}{2})|^2dx}.
\end{eqnarray}
At the same time, by \eqref{truncation} and $\phi:=\sigma_0\cdot u$, we get
\begin{eqnarray*}
\frac{\int_{\Omega\cap B_{R_0}}|\phi(x,T-\varepsilon)|^2dx}{\int_{\Omega\cap  B_{(1+\delta_g)R_g}}|\phi(x,T-\frac{\varepsilon}{2})|^2dx}
\leq \frac{\int_{\Omega\cap B_{R_0}}|u(x,T-\varepsilon)|^2dx}{\int_{\Omega\cap B_{(1+\delta_g)R_g}}|u(x,T-\frac{\varepsilon}{2})|^2dx}.
\end{eqnarray*}
Then, by Theorem \ref{proposition3.2},  Lemma \ref{theta} and the fact that $\varepsilon\in(0, \theta)$,
we see that
\begin{eqnarray*}
\frac{\int_{\Omega\cap B_{R_0}}|\phi(x,T-\varepsilon)|^2dx}{\int_{\Omega\cap  B_{(1+\delta_g)R_g}}|\phi(x,T-\frac{\varepsilon}{2})|^2dx}
\leq \frac{e^{L_1M^2T}\int_{\Omega} |u(x,0)|^2dx}{\int_{\Omega\cap B_{(1+\delta_g)R_g}}|u(x,T-\frac{\varepsilon}{2})|^2dx}
\leq e^{L_1M^2T}e^{\frac{L_2}{\theta}}.
\end{eqnarray*}
This, together with \eqref{4.23} and \eqref{3.32}, yields
\begin{eqnarray*}
e^{\frac{\varepsilon}{2(\lambda+\varepsilon)}\exp(-C_2M^2\varepsilon) \lambda N_{\lambda}(T)}\leq e^{(C_3M^2+C_2M^2\varepsilon
+Q(\theta,\varepsilon, \lambda))\frac{\varepsilon}{2}}e^{\frac{(1+\delta_g)^2R_g^2}{2\varepsilon}}e^{L_1M^2T}e^{\frac{L_2}{\theta}},
\end{eqnarray*}
which gives \eqref{4.24aa}.

\noindent{\it Step 4.  We  prove  (\ref{4.8aa}) via tuning  parameters.}\

Let $\varepsilon=k\theta$ and $\lambda=\mu\varepsilon$, where
 $k:=\min\{\frac{L_6}{2L_2, },\frac{1}{2}\}$ and
  $\mu\in(0,1)$ will be determined later. Then we have
   $L_2-\frac{L_6}{k(1+\mu)}<0$. This, together with \eqref{3.34}
    (where $\varepsilon=k\theta$ and $\lambda=\mu\varepsilon$)
    and the fact that $0<\varepsilon<\theta< \min\{ 1, \frac{T}{2}\}$, indicates
\begin{eqnarray*}
Q(\theta,\varepsilon, \lambda)\leq 12C_1M^2(1+\frac{1}{T})e^{L_1M^2T}.
\end{eqnarray*}
Then, by \eqref{4.24aa} and  the fact that $\lambda=\mu\varepsilon$, we see
\begin{eqnarray*}
&&\varepsilon\lambda N_{\lambda}(T)\leq 2(\mu+1)e^{C_2M^2\varepsilon}
\bigg[\frac{C_3}{2}M^2\varepsilon^2+\frac{C_2}{2}M^2\varepsilon^3 \\
&+&6C_1M^2(1+\frac{1}{T})e^{L_1M^2T}\varepsilon^2+\frac{(1+\delta_g)^2R_g^2}{2}+L_1M^2T\varepsilon+\frac{L_2\varepsilon}{\theta}\bigg].
\end{eqnarray*}
Since $\frac{\varepsilon}{\theta}=k$ and  $\varepsilon, k,\mu\in(0,1)$, we have
\begin{eqnarray*}
\varepsilon\lambda N_{\lambda}(T)\leq 4e^{C_2M^2}\!
\bigg[\frac{C_3}{2}M^2\!+\!\frac{C_2}{2}M^2
+6C_1M^2(1+\frac{1}{T})e^{L_1M^2T}\!+\!\frac{(1+\delta_g)^2R_g^2}{2}+L_1M^2T+L_2\bigg].
\end{eqnarray*}
From the above, we can find a  constant $C_0>1$, depending on $\Omega, r, R_g, \delta_g, M$, and $T$, so that
\begin{eqnarray}\label{4.13}
\frac{16\lambda}{r^2}\left(\lambda N_\lambda(T)+\frac{n}{4}\right)=\frac{16\mu\varepsilon}{r^2}\left(\lambda N_\lambda(T)+\frac{n}{4}\right)
\leq \frac{16}{r^2}\mu\left(\varepsilon\lambda N_{\lambda}(T)+\frac{n}{4}\right)\leq \mu C_0.
\end{eqnarray}
Choosing $\mu=\frac{1}{2C_0}\in(0,1)$ in (\ref{4.13}), we obtain
\begin{eqnarray}\label{3.37aa}
\frac{16\lambda}{r^2}\left(\lambda N_\lambda(T)+\frac{n}{4}\right)\leq\frac{1}{2}.
\end{eqnarray}\

Next, since $0<r<R_g$ and $B_r:=B(x_g,r)\subset\Omega$, we find
\begin{eqnarray}\label{4.24}
&&\int_{\Omega\cap B_{R_0}}|\phi(x,T)|^2e^{-\frac{|x-x_g|^2}{4\lambda}}dx\crr\disp
&=&\int_{\Omega\cap B_{R_0}\setminus B_r}|\phi(x,T)|^2e^{-\frac{|x-x_g|^2}{4\lambda}}dx+\int_{B_r}|\phi(x,T)|^2
e^{-\frac{|x-x_g|^2}{4\lambda}}dx\crr\disp
&\leq&\frac{1}{r^2}\int_{\Omega\cap B_{R_0} \setminus B_r}|x-x_g|^2|\phi(x,T)|^2e^{-\frac{|x-x_g|^2}{4\lambda}}dx
+\int_{B_r}|\phi(x,T)|^2e^{-\frac{|x-x_g|^2}{4\lambda}}dx.
\end{eqnarray}
Meanwhile, it follows from Lemma \ref{lemma 4.3} that
\begin{eqnarray*}
&&\int_{\Omega\cap B_{R_0}}|x-x_g|^2|\phi(x,T)|^2e^{-\frac{|x-x_g|^2}{4\lambda}}dx\crr\disp
&\leq& 8\lambda\bigg(2\lambda\int_{\Omega\cap B_{R_0}}|\nabla \phi(x,T)|^2 e^{-\frac{|x-x_g|^2}{4\lambda}}dx+\frac{n}{2}\int_{\Omega\cap B_{R_0}}|\phi(x,T)|^2e^{-\frac{|x-x_g|^2}{4\lambda}}dx\bigg)\crr\disp
&\leq& 8\lambda\bigg(2\lambda N_{\lambda}(T)\int_{\Omega\cap B_{R_0}}|\phi(x,T)|^2e^{-\frac{|x-x_g|^2}{4\lambda}}dx+\frac{n}{2}\int_{\Omega\cap B_{R_0}}|\phi(x,T)|^2e^{-\frac{|x-x_g|^2}{4\lambda}}dx\bigg).
\end{eqnarray*}
Combining the above  with \eqref{4.24} yields
\begin{eqnarray*}
&&\int_{\Omega\cap B_{R_0}}|\phi(x,T)|^2e^{-\frac{|x-x_g|^2}{4\lambda}}dx\crr\disp
&\leq&\frac{16\lambda}{r^2}\big(\lambda N_\lambda(T)+\frac{n}{4}\big)\int_{\Omega\cap B_{R_0}}|\phi(x,T)|^2e^{-\frac{|x-x_g|^2}{4\lambda}}dx+\int_{B_r}|\phi(x,T)|^2e^{-\frac{|x-x_g|^2}{4\lambda}}dx.
\end{eqnarray*}
This, along with \eqref{3.37aa}, implies
\begin{eqnarray}\label{3.39}
\int_{\Omega\cap B_{R_0}}|\phi(x,T)|^2e^{-\frac{|x-x_g|^2}{4\lambda}}dx
\leq2\int_{B_r}|\phi(x,T)|^2e^{-\frac{|x-x_g|^2}{4\lambda}}dx.
\end{eqnarray}

Now, we are going to prove  (\ref{4.8aa}). One can easily check from \eqref{truncation} and $\phi:=\sigma_0\cdot u$
that
\begin{eqnarray*}
&&\int_{\Omega\cap B_{R_g}}|u(x,T)|^2dx \leq e^{\frac{R_g^2}{4\lambda}}\int_{\Omega\cap B_{R_0}}|\phi(x,T)|^2e^{-\frac{|x-x_g|^2}{4\lambda}}dx.
\end{eqnarray*}
This, along with  \eqref{3.39}   and  \eqref{truncation}, shows
\begin{eqnarray*}
\int_{\Omega\cap B_{R_g}}|u(x,T)|^2dx \leq 2e^{\frac{R_g^2}{4\lambda}}\int_{B_r}|\phi(x,T)|^2e^{-\frac{|x-x_g|^2}{4\lambda}}dx \leq 2e^{\frac{R_g^2}{4\lambda}}\int_{B_r}|u(x,T)|^2dx.
\end{eqnarray*}
Since $\lambda=\mu\varepsilon=\mu k \theta$ (where $k=\min\{\frac{L_6}{2L_2},\frac{1}{2}\}$ and  $\mu=\frac{1}{2C_0}$), the above leads to
\begin{eqnarray*}
\int_{\Omega\cap B_{R_g}}|u(x,T)|^2dx \leq 2e^{\frac{1}{\theta}\frac{R_g^2}{4\mu k}}\int_{B_r}|u(x,T)|^2dx.
\end{eqnarray*}
This, along with \eqref{3.8,1-25} and \eqref{3.9,1-25}, indicates
\begin{eqnarray}\label{3.37}
\int_{\Omega\cap B_{R_g}}|u(x,T)|^2dx
\leq 2\bigg(e^{L_4M^2T}e^{L_5(1+\frac{1}{T})}\frac{\mathbb{E}}{\int_{\Omega\cap B_{R_g}}|u(x,T)|^2dx}\bigg)^{\frac{L_3R_g^2}{4\mu k}}\int_{B_r}|u(x,T)|^2dx.
\end{eqnarray}
 Meanwhile,  according to Theorem \ref{proposition3.2},
\begin{eqnarray*}
\mathbb{E}\leq (1+Te^{L_1M^2T})\int_{\Omega}|u(x,0)|^2dx\leq (1+T)e^{L_1M^2T}\int_{\Omega}|u(x,0)|^2dx.
\end{eqnarray*}
This, together with \eqref{3.37}, implies
\begin{eqnarray*}
\int_{\Omega\cap B_{R_g}}|u(x,T)|^2dx
\leq 2\bigg(D\frac{\int_{\Omega}|u(x,0)|^2dx}{\int_{\Omega\cap B_{R_g}}|u(x,T)|^2dx}\bigg)^{\frac{L_3R_g^2}{4\mu k}}\int_{B_r}|u(x,T)|^2dx,
\end{eqnarray*}
where $D=(1+T)e^{(L_1+L_4)M^2T+L_5(1+\frac{1}{T})}$.
The above  leads to \eqref{4.8aa} with
$\gamma=\frac{L_3R_g^2}{4\mu k+L_3R_g^2}$ (which depends only on $\Omega, r, R_g, \delta_g, M$, and $T$).
This completes the proof of Theorem \ref{lemma4.4}.
\end{proof}
\begin{remark}\label{remark3.3,1-31}
It deserves mentioning that   \eqref{4.8aa} is a local interpolation inequality of the boundary case
for the system (\ref{1.1}).

By the same argument used in Theorem \ref{lemma4.4}, we can verify the local interpolation inequality of the interior case for
the system (\ref{1.1}): {\it for each
 $p_0\in\Omega$, there is a positive number $r_0$, with $B(p_0, 3r_0)\subseteq\Omega$. (Notice that
the open ball $B(p_0, 3r_0)$ is star-shaped with the center $p_0$.)
 Then,  there are  two constants
$D=D(r_0, \Omega, T, M)>0$ and $\gamma=\gamma(r_0, \Omega, T, M)\in(0,1)$ so that
\begin{eqnarray*}
 \int_{B(p_0, r_0)}|u(x,T)|^2dx
\leq \bigg(D\int_{\Omega}|u_0(x)|^2dx\bigg)^\gamma\bigg(2\int_{B(p_0, r_0/2)}|u(x,T)|^2dx\bigg)^{1-\gamma}.
\end{eqnarray*}}
We  omit the detailed proof.
\end{remark}

\section{Proof of the main results}
This section first proves  Theorem \ref{theorem},
 then it gives a
   qualitative unique continuation property for the system (\ref{1.1}), as  a consequence of
 Theorem \ref{theorem}.
\subsection{Proof of Theorem \ref{theorem}}

\begin{proof}[Proof of Theorem \ref{theorem}]
The proof will be organized in two steps.

\noindent{\it Step 1.  We  prove  (\ref{1.4}), with the help of Thoerem \ref{lemma4.4}.}

Since $\omega$ is a nonempty open subset of $\Omega$,  we
can find $x_0\in \omega$ and $r>0$ so that the open ball $B(x_0,r)$ belongs to $\omega$.
We are going to split the proof into two sub-steps.

\noindent{\it Sub-step 1.1. We  deal with the boundary of $\Omega$.}

 Since $\Omega$ is bounded domain with a $C^2$ boundary $\partial\Omega$,
it follows from \eqref{3.1} that
\begin{eqnarray*}
\partial\Omega\subset\cup_{g\in\partial\Omega}B(x_g,R_g)\;  \textrm{ and } \;\Omega\cap B(x_g,(1+2\delta_g)R_{g}) \textrm{ is star-shaped with } x_{g},
\end{eqnarray*}
where the triplet $(x_{g}, R_{g}, \delta_g)\in\Omega\times\mathbb{R}^{+}\times(0,1]$ corresponding to $g\in\partial\Omega$ is given by \eqref{3.1}.  Then by the compactness of $\partial\Omega$, we
can find a finite set of triplets $(x_i, R_i, \delta_i)\in\Omega\times\mathbb{R}^{+}\times(0,1]$
$(i=1,2,\ldots,m_1)$ such that $\partial\Omega\subset\cup_{i=1,2,\ldots,m_1}B(x_i,R_i)$
and such that each  $\Omega\cap B(x_i,(1+2\delta_i)R_i)$ is star-shaped with respect to  $x_i$. Let
\begin{eqnarray}\label{4.1aa}
\Theta_1=\cup_{i=1,2,\ldots,m_1}\Omega\cap B(x_i,R_i).
\end{eqnarray}

We  claim that there exist two constants $D=D(\Theta_1, \Omega, r, M, T)>0$ and $\gamma_1=\gamma_1(\Theta_1, \Omega, r, M, T)\in(0,1)$ such that
\begin{eqnarray}\label{4.1}
 \int_{\Theta_1}|u(x,T)|^2dx
\leq D\bigg(\int_{\Omega}|u_0(x)|^2dx\bigg)^{\gamma_1}\bigg(\int_{B(x_0,r)}|u(x,T)|^2dx\bigg)^{1-\gamma_1}.
\end{eqnarray}

In fact, for each $i\in\{1,2,\ldots,m_1\}$, we can choose $\rho_i\in(0,R_i)$ and finitely
many points $q_{i,1}, q_{i,2},  \ldots, q_{i,d_i}\in\Omega$ so that
\begin{eqnarray}\label{chain}
\begin{cases}
x_i=q_{i,1}; \\
 B(q_{i,j},\rho_i/2)\subset B(q_{i,j+1},\rho_i),\: \forall \: j=1,2,\ldots, d_i-1;\\
B(q_{i,d_i},\rho_i)\subset B(x_0,r);\\
B(q_{i,j},3\rho_i)\subset\Omega,  \: \forall \: j=1,2,\ldots, d_i,
\end{cases}
\end{eqnarray}
which forms a chain of balls along a curve connecting  $q_{i,1}$ with $q_{i,d_i}$. Then, it follows from Theorem \ref{lemma4.4} that
there are constants $D_{i,1}>0$ and $\alpha_{i,1}\in(0,1)$ (which depend on $\rho_i, R_i, \delta_i, \Omega, M$, and $T$)
so that
\begin{eqnarray*}
 \int_{\Omega\cap B(x_i, R_i)}|u(x,T)|^2dx
\leq D_{i,1}\bigg(\int_{\Omega}|u_0(x)|^2dx\bigg)^{\alpha_{i,1}}\bigg(\int_{B(x_i, \: \rho_i/2)}|u(x,T)|^2dx\bigg)^{1-\alpha_{i,1}},
\end{eqnarray*}
which, along with the first fact   in \eqref{chain}, yields
\begin{eqnarray}\label{4.5bb}
 \int_{\Omega\cap B(x_i, R_i)}|u(x,T)|^2dx
&\leq& D_{i,1}\bigg(\int_{\Omega}|u_0(x)|^2dx\bigg)^{\alpha_{i,1}}\bigg(\int_{B(q_{i,1}, \: \rho_i/2)}|u(x,T)|^2dx\bigg)^{1-\alpha_{i,1}}.
\end{eqnarray}

Now, we will propagate the  interpolation inequality \eqref{4.5bb}
along the chain of balls \eqref{chain}. First, combining \eqref{4.5bb} with the second fact  in \eqref{chain}
leads to
\begin{eqnarray}\label{4.4ee}
 \int_{\Omega\cap B(x_i, R_i)}|u(x,T)|^2dx
\leq D_{i,1}\bigg(\int_{\Omega}|u_0(x)|^2dx\bigg)^{\alpha_{i,1}}\bigg(\int_{B(q_{i,2}, \: \rho_i)}|u(x,T)|^2dx\bigg)^{1-\alpha_{i,1}}.
\end{eqnarray}
Next, we deal with
 the term $\int_{B(q_{i,2}, \: \rho_i)}|u(x,T)|^2dx$ in \eqref{4.4ee} in the following manner:
 As mentioned in Remark \ref{remark3.3,1-31}. (Notice the fourth fact in  \eqref{chain}.)
 there are  two constants $C_{i,2}=C_{i,2}(\rho_i, \Omega, M, T)>0$ and $\beta_{i,2}=\beta_{i,2}(\rho_i, \Omega, M, T)\in(0,1)$ so that
\begin{eqnarray}\label{4.6,1-31}
\int_{B(q_{i,2}, \: \rho_i)}|u(x,T)|^2dx
\leq C_{i,2}\bigg(\int_{\Omega}|u_0(x)|^2dx\bigg)^{\beta_{i,2}}\bigg(\int_{B(q_{i,2}, \: \rho_i/2)}|u(x,T)|^2dx\bigg)^{1-\beta_{i,2}}.
\end{eqnarray}
Combining   \eqref{4.4ee} with \eqref{4.6,1-31} leads to
\begin{eqnarray*}
 && \int_{\Omega\cap B(x_i, R_i)}|u(x,T)|^2dx
\leq  D_{i,1}\bigg(\int_{\Omega}|u_0(x)|^2dx\bigg)^{\alpha_{i,1}}   \crr\disp
&\times& \bigg(C_{i,2}\big(\int_{\Omega}|u_0(x)|^2dx\big)^{\beta_{i,2}}\big(\int_{B(q_{i,2}, \: \rho_i/2)}|u(x,T)|^2dx\big)^{1-\beta_{i,2}}\bigg)^{1-\alpha_{i,1}}\crr\disp
&=&   D_{i,2}\bigg(\int_{\Omega}|u_0(x)|^2dx\bigg)^{\alpha_{i,2}}  \bigg(\int_{B(q_{i,2}, \: \rho_i/2)}|u(x,T)|^2dx\bigg)^{1-\alpha_{i,2}},
\end{eqnarray*}
where $D_{i,2}=D_{i,1}\cdot C_{i,2}^{1-\alpha_{i,1}}>0$ and $\alpha_{i,2}=\alpha_{i,1}+\beta_{i,2}(1-\alpha_{i,1})\in(0,1)$.
Propagating interpolation inequalities  finite times along the chain of balls \eqref{chain}, we
can find  constants $D_i=D_i(\rho_i, R_i, \delta_i, \Omega, M, T)>0$ and $\alpha_i=\alpha_i(\rho_i, R_i, \delta_i, \Omega, M, T)\in(0,1)$ such that
\begin{eqnarray*}
 \int_{\Omega\cap B(x_i, R_i)}|u(x,T)|^2dx
\leq D_i\bigg(\int_{\Omega}|u_0(x)|^2dx\bigg)^{\alpha_i}\!\bigg(\int_{B(q_{i,d_i-1}, \: \rho_i/2)}|u(x,T)|^2dx\bigg)^{1-\alpha_{i}}.
\end{eqnarray*}
This, along with the second and the third fact of \eqref{chain}, yields that when $i\in\{1,\dots,m_1\}$,
\begin{eqnarray}\label{4.6}
 \int_{\Omega\cap B(x_i, R_i)}|u(x,T)|^2dx
\leq D_i\bigg(\int_{\Omega}|u_0(x)|^2dx\bigg)^{\alpha_i}\!\bigg(\int_{B(x_0, \: r)}|u(x,T)|^2dx\bigg)^{1-\alpha_i}.
\end{eqnarray}
Let
\begin{eqnarray}\label{4.8bb}
\gamma_1:=\max\{\alpha_i\:|\: i=1,2,\ldots, m_1\}.
\end{eqnarray}
We can easily check that $\gamma_1\in(0,1)$. This, together with \eqref{4.6}, \eqref{4.8bb}  and \eqref{3.8aaa}, implies that for each $i\in\{1,\dots,m_1\}$,
\begin{eqnarray}\label{4.9aa}
 &&\int_{\Omega\cap B(x_i, R_i)}|u(x,T)|^2dx\leq D_i\bigg(\int_{\Omega}|u_0(x)|^2dx\bigg)^{\alpha_i}\crr\disp
&\times& \bigg(\int_{B(x_0, \: r)}|u(x,T)|^2dx\bigg)^{1-\gamma_1}\bigg(e^{L_1M^2T}\int_{\Omega}|u_0(x)|^2dx\bigg)^{\gamma_1-\alpha_{i}}\crr\disp
&=&\bar{D}_i\bigg(\int_{\Omega}|u_0(x)|^2dx\bigg)^{\gamma_1}\!\bigg(\int_{B(x_0, \: r)}|u(x,T)|^2dx\bigg)^{1-\gamma_1},
\end{eqnarray}
where $\bar{D}_i=D_i\cdot e^{(\gamma_1-\alpha_{i})L_1M^2T}>0$.
Finally, by  \eqref{4.1aa} and \eqref{4.9aa}, we get \eqref{4.1},
 with $D:=\sum_{i=1}^{m_1}\bar{D}_i$ and $\gamma_1$ given by \eqref{4.8bb}.

\noindent{\it Sub-step 1.2. We deal with the interior of $\Omega$.}

 It is obvious that there exists a compact subset $\Theta_2\subset\Omega$ such that $\Omega\subseteq \Theta_1\cup\Theta_2$.
By the compactness of $\Theta_2$, there is a constant $R>0$ and finitely many points $y_1, y_2\ldots, y_{m_2}\in\Omega$ such that
$\Theta_2\subset\cup_{i=1,2,\ldots,m_2}B(y_i,R)$ and $B(y_i,3R)\subset\Omega$ for each $i\in\{1, 2, \ldots, m_2\}$. Then,
by the same method used to prove \eqref{4.1}, we can find
constants $D=D(\Theta_2,\Omega, r,M, T)>0$ and $\gamma_2=\gamma_2(\Theta_2,\Omega, r, M, T)\in(0,1)$ so that
\begin{eqnarray}\label{4.16}
 \int_{\Theta_2}|u(x,T)|^2dx
\leq D\bigg(\int_{\Omega}|u_0(x)|^2dx\bigg)^{\gamma_2}\bigg(\int_{B(x_0,r)}|u(x,T)|^2dx\bigg)^{1-\gamma_2}.
\end{eqnarray}

Finally, by (\ref{4.1}), (\ref{4.16}), (\ref{M}), and Lemma \ref{lemma2.1}, we  obtain  (\ref{1.4}).\

\noindent{\it Step 2.   We  prove \eqref{1.5}.}

The proof will also be split into two sub-steps.

\noindent{\it Sub-step 2.1.  We  prove that   $u(\cdot,t)\neq0$ for each $t\in[0,T]$.}

By contradiction, we suppose that  $u(\cdot,t)=0$ for some $t\in(0,T]$.
Then by the assumption that $u_0\neq 0$, we have that  $0<T_0\leq T$, where
 \begin{eqnarray*}
T_0:=\inf\{t\in(0,T]\:|\: u(\cdot,t)=0\}.
 \end{eqnarray*}
 This, along with  the continuity  of $u$, yields
  \begin{eqnarray}\label{4.11aa}
 u(\cdot, T_0)=0  \textrm{  and  }  u(\cdot, t)\neq0,  \textrm{  for each  } t\in[0,T_0).
 \end{eqnarray}
Let
   \begin{eqnarray}\label{4.10,1.31}
 \zeta(t) := {\|u(\cdot, t)\|_2^2\over \|u(\cdot, t)\|^2_{H^{-1}}},\; t\in[0,T_0),
 \end{eqnarray}
 where  $\|\cdot\|_{H^{-1}}$ is the norm of  $H^{-1}(\Omega)$.

We now claim the following backward uniqueness estimate for $u$:
\begin{eqnarray}\label{4.29}
 \|u_0\|_{H^{-1}}^2\le \exp(2e^{C(\Omega)M^2T} \left(\zeta(0) + C(\Omega)M\sqrt{\zeta(0)}\right)T)\|u(\cdot,t)\|_{H^{-1}}^2, \textrm{ for each } t\in[0,T_0).
\end{eqnarray}
 To this end, by multiplying  the first equation of (\ref{1.1}) by $u$ and $(-\triangle)^{-1}u$ respectively, and integrating them over $\Omega$,  we obtain that for each $t\in[0,T]$,
\begin{eqnarray}\label{3.6}
\begin{cases}
{1\over 2}{d\over dt}\|u(\cdot, t)\|_2^2 + \|u(\cdot, t)\|^2_{H^1_0} = \langle-\nabla(u(\cdot, t)\nabla v(\cdot, t)), u(\cdot, t)\rangle,\\
{1\over 2}{d\over dt}\|u(\cdot, t)\|_{H^{-1}}^2 + \|u(\cdot, t)\|_2^2 = \langle-\nabla(u(\cdot, t)\nabla v(\cdot, t)), (-\triangle)^{-1}u(\cdot, t)\rangle_{H^{-1}, H_0^1}.
\end{cases}
\end{eqnarray}
 (Here, $\|\cdot\|_{H^{1}_0}$ denotes the norm of the space  $H^{1}_0(\Omega)$, $\langle \cdot,\cdot\rangle$
 is the inner product in  $L^2(\Omega)$,
and $\langle\cdot,\cdot\rangle_{H^{-1}, H_0^1}$ stands for  the pair between $H^{-1}(\Omega)$ and $H^{1}_0(\Omega)$.)
Write
\begin{eqnarray}\label{4.12,1.31}
f(x,t) := -\nabla(u(x,t) \nabla v(x,t) ),\;\;x\in\Omega,\;t\in[0,T].
\end{eqnarray}
By \eqref{4.12,1.31}, \eqref{4.3bb}, and \eqref{3.6},
after direct computations, we get that
 \begin{eqnarray}\label{4.29aa}
\|f\|_{H^{-1}} &=& \|\nabla (u\nabla v)\|_{H^{-1}}\le  \| u \nabla v\|_{[L^2(\Omega)]^n} \crr\disp
 &\le& \|u\|_{2}\|\nabla v\|_{[L^\infty(\Omega)]^n}  \le C(\Omega)M\|u\|_{2}, \;\; t\in[0,T];
\end{eqnarray}
\begin{eqnarray}\label{4.11}
\zeta'(t)&=&\frac{2}{\|u\|_{H^{-1}}^4}\big(\langle f, u\rangle\|u\|^2_{H^{-1}}-\|u\|^2_{H_0^1}\|u\|^2_{H^{-1}}\crr\disp
& - &\langle f, (-\Delta)^{-1}u\rangle_{H^{-1}, H_0^1}\|u\|_2^2 +\|u\|_2^4\big), \;\; t\in[0,T_0);
\end{eqnarray}
\begin{eqnarray}\label{4.15,1.31}
&&\|u\|_2^4 - \|u\|_2^2\langle f, (-\triangle)^{-1}u\rangle_{H^{-1}, H_0^1}  \nonumber\\
&=& |\langle\triangle u+ {f/2}, (-\triangle)^{-1}u\rangle_{H^{-1}, H_0^1}|^2-|\langle{f/2}, (-\triangle)^{-1}u\rangle_{H^{-1}, H_0^1}|^2 \nonumber \\
&\leq& \|\triangle u+ {f/2}\|_{H^{-1}}^2\cdot\|(-\triangle)^{-1}u\|_{H_0^1}^2-|\langle{f/2}, (-\triangle)^{-1}u\rangle_{H^{-1}, H_0^1}|^2 \nonumber \\
&=&\big(\|u\|^2_{H_0^1}+\|{f/2}\|_{H^{-1}}^2-\langle f,u\rangle\big)\|u\|^2_{H^{-1}}-|\langle{f/2}, (-\triangle)^{-1}u\rangle_{H^{-1}, H_0^1}|^2, \;\;  t\in[0,T].
\end{eqnarray}
By \eqref{4.11} and \eqref{4.15,1.31}, we see
\begin{eqnarray*}
\zeta'(t)\le {2\over \|u\|^2_{H^{-1}}}\|{f/2}\|_{H^{-1}}^2,\;\;t\in[0,T_0),
\end{eqnarray*}
which, together with  \eqref{4.29aa}, yields
 \begin{eqnarray}\label{4.27}
 \zeta(t)\le e^{C(\Omega)M^2t}\zeta(0),\;\;t\in[0,T_0).
 \end{eqnarray}
Now, by  the second equation in (\ref{3.6}), \eqref{4.27}, and  \eqref{4.29aa},  we have
   \begin{eqnarray*}
 &&0\le {1\over 2}{d\over dt}\|u\|_{H^{-1}}^2 +\zeta(t) \|u\|_{H^{-1}}^2 + |\langle f, (-\triangle)^{-1}u\rangle_{H^{-1}, H_0^1}|\\
 &&\quad \le {1\over 2}{d\over dt}\|u\|_{H^{-1}}^2 +\zeta(t)\|u\|_{H^{-1}}^2 + \|f\|_{H^{-1}}\| (-\triangle)^{-1}u\|_{H^{1}_0}\\
 &&\quad\le {1\over 2}{d\over dt}\|u\|_{H^{-1}}^2 +\zeta(0)e^{C(\Omega)M^2t} \|u\|_{H^{-1}}^2 + C(\Omega)M\|u\|_{2}\|u\|_{H^{-1}},\;\;t\in[0,T_0),
 \end{eqnarray*}
 which, along with \eqref{4.10,1.31} and \eqref{4.27}, shows
  \begin{eqnarray*}
 &&0\leq{1\over 2}{d\over dt}\|u\|_{H^{-1}}^2 +e^{C(\Omega)M^2T} \big(\zeta(0) + C(\Omega)M\sqrt{\zeta(0)}\big)\|u\|_{H^{-1}}^2,\;\;t\in[0,T_0).
 \end{eqnarray*}
  Multiplying the above by $\exp\big(2e^{C(\Omega)M^2T} \big(\zeta(0) + C(\Omega)M\sqrt{\zeta(0)}\big)t\big)$, and then
  integrating it over  $(0, t)$, where $t\in[0,T_0)$, we obtain  (\ref{4.29}).\\

Next, it follows  from (\ref{4.29}) that when $t\in[0,T_0)$,
 \begin{eqnarray*}
&& \frac{\|u_0\|_{2}^2}{\|u(\cdot, t)\|_{2}^2}\leq \frac{\|u_0\|_{H^{-1}}^2}{\|u(\cdot, t)\|_{H^{-1}}^2}\zeta(0)\nonumber\\
&\leq& \zeta(0)\exp\left(2e^{C(\Omega)M^2T} (\zeta(0) + C(\Omega)M\sqrt{\zeta(0)})T\right).
 \end{eqnarray*}
  Since $\zeta(0)\geq1$,  we have that $\sqrt{\zeta(0)}\leq\zeta(0)$  and  $\zeta(0)<\exp\big(\zeta(0)\big)$.
 These show that
  \begin{eqnarray}\label{4.20}
 \frac{\|u_0\|_{2}^2}{\|u(\cdot, t)\|_{2}^2}\leq  \exp\left(C(\Omega)(1+MT)e^{C(\Omega)M^2T} \zeta(0)\right),\;\;t\in[0,T_0).
 \end{eqnarray}

Meanwhile, by $(ii)$ in Remark  \ref{remark2.3,2-11} and the assumption that $u_0\in L^{\infty}(\Omega)$, we see that $u\in C([0,T];L^2(\Omega))$.
 This, along with \eqref{4.11aa}, implies that
 \begin{eqnarray*}
 \lim_{t\rightarrow T_0^{-}}\|u(\cdot, t)\|_{2}^2=\|u(\cdot, T_0)\|_{2}^2=0,
 \end{eqnarray*}
 which, together with the assumption that $u_0\neq 0$, contradicts \eqref{4.20}.
 Hence, we finish the proof of {\it Sub-step 2.1}.

\noindent{\it Sub-step 2.2. We  prove \eqref{1.5}.}

According to  {\it Sub-step 2.1},  the function $t\rightarrow{\|u(\cdot, t)\|_2^2\over \|u(\cdot, t)\|^2_{H^{-1}}}$ (see \eqref{4.10,1.31}) is well defined over $[0,T]$. We still use
$\zeta(\cdot)$ to denote this function on $[0,T]$.
 Then by the same method in  the proof of  \eqref{4.20}, we can verify that
  \begin{eqnarray*}
 \|u_0\|_{2}^2\leq  \exp\big(C(\Omega)(1+MT)e^{C(\Omega)M^2T} \zeta(0)\big)\|u(\cdot, T)\|_{2}^2,
 \end{eqnarray*}
which, together with (\ref{1.4}), gives (\ref{1.5}).

 Hence, we  finish the proof of  Theorem \ref{theorem}.
\end{proof}

\subsection{Consequence of  Theorem \ref{theorem}}
This subsection presents a qualitative unique continuation property  for the system (\ref{1.1}), which is a consequence of
 Theorem \ref{theorem}.

\begin{corollary}\label{theorem1.3}
Let  $u_0\in L^p(\Omega)$ with $n<p\leq+\infty$ and let $\omega$ be a nonempty open subset of $\Omega$. Suppose that  $(u,v)$ is the solution to the system (\ref{1.1}) over $[0,T]$ for some $T>0$.
Then
   \begin{eqnarray}\label{4.17}
u=0\;\mbox{over}\;\Omega\times[0, T] \: \textrm{ and } v=0\;\mbox{over}\;\Omega\times(0, T],
 \end{eqnarray}
 provided that either $u(\cdot,T)=0$ over $\omega$ or $v(\cdot,T)=0$ over $\omega$.
 \end{corollary}

\begin{proof}
We organize the proof in two steps.

\noindent{\it Step 1.  We show the corollary  when $p=+\infty$.}

In the first case that
$u(\cdot,T)=0$ over $\omega$,
we can apply \eqref{1.5} to see that
$u_0=0$,
 which, along with $(iii)$ of Remark \ref{remark2.3,2-11},
 leads to \eqref{4.17} in this case.

We next consider the second case that $v(\cdot,T)=0$ over $\omega$. Multiplying the second equation of (\ref{1.1}) by a test function  $\chi\in C_0^\infty(\Omega)$ with $\textrm{supp}\:\chi\subset\omega$, we find
\begin{eqnarray*}
&&\int_{\Omega}u(\cdot,T)\chi dx=\frac{1}{b}\int_{\Omega}[-\triangle v(\cdot,T)+av(\cdot,T)]\chi dx  \crr\disp
&=&\frac{1}{b}\big[\int_{\Omega}av(\cdot,T)\chi dx+ \int_{\Omega} v(\cdot,T)(-\triangle\chi)dx\big]=0,
\end{eqnarray*}
which yields that  $u(\cdot,T)=0$ over $\omega$, i.e.,
we return to the first case.
Consequently,  \eqref{4.17} is true for the second case.

\noindent{\it Step 2.  We show the corollary  when  $n<p<+\infty$.}

 Arbitrarily fix $\epsilon\in(0, T)$. We define two functions
 $u_{\epsilon}$ and $v_{\epsilon}$
  on
$\Omega\times[0,T-\epsilon]$ by setting $u_{\epsilon}(x,t):=u(x,t+\epsilon)$ and $v_{\epsilon}(x,t):=v(x,t+\epsilon)$, respectively. It is
obvious that $u_{\epsilon}$ and $v_{\epsilon}$ satisfy
 \begin{eqnarray*}
\begin{cases}
u_{\epsilon t}-\triangle u_\epsilon(x,t)+\nabla\cdot(u_\epsilon(x,t)\nabla v_\epsilon(x,t))=0,& \textrm{ in } \Omega\times(0,T-\epsilon],\\
-\triangle v_{\epsilon}(x,t)+av_{\epsilon}(x,t)-bu_{\epsilon}(x,t)=0,& \textrm{ in } \Omega\times(0,T-\epsilon],\\
u_{\epsilon}(x, t)=0, \: v_{\epsilon}(x,t)=0,&  \textrm{ on } \partial\Omega\times(0,T-\epsilon],\\
u_{\epsilon}(x,0)=u(x,\epsilon),     &   \textrm{ in } \Omega.
\end{cases}
\end{eqnarray*}
Since $u_0\in L^p(\Omega)$ $(n<p<+\infty)$, it follows from the conclusion $(ii)$ in Theorem \ref{theorem4.2}  that $u_{\epsilon}(\cdot,0)=u(\cdot,\epsilon)\in L^\infty(\Omega)$ and $(u_\epsilon, v_\epsilon)\in [L^\infty(0,T-\epsilon; L^\infty(\Omega))]^2$.
Thus we can apply Theorem \ref{theorem} (where $(u,v)$ is replaced by $(u_{\epsilon},v_{\epsilon})$)
 to see what follows:

 \noindent
$(i)$ There are constants  $\gamma= \gamma(\Omega, \omega, \|u_{\epsilon}(\cdot,0)\|_{\infty}, T-\epsilon)\in (0,1)$ and $D=D(\Omega, \omega, \|u_{\epsilon}(\cdot,0)\|_{\infty}, T-\epsilon)>0$
  so that
\begin{eqnarray*}
\int_{\Omega}|u_\epsilon(x,T-\epsilon)|^2dx\leq D\!\left(\int_{\Omega}|u_\epsilon(x,0)|^2dx\right)^{1-\gamma}\!\!\left(\int_{\omega}|u_\epsilon(x,T-\epsilon)|^2dx\right)^\gamma.
\end{eqnarray*}

\noindent
$(ii)$ When $u(\cdot,\epsilon)\neq 0$, there is   $C=C(\Omega,\omega, \|u_{\epsilon}(\cdot,0)\|_{\infty}, T-\epsilon)>0$ so that
\begin{eqnarray*}
\int_{\Omega}|u_\epsilon(x,0)|^2dx\leq C\exp\left(C{\|u_\epsilon(x,0)\|_{L^2(\Omega)}^2\over \|u_\epsilon(x,0)\|_{H^{-1}(\Omega)}^2}\!\right)
\!\times\! \int_{\omega}(|u_{\epsilon}(x,T-\epsilon))|^2)dx.
\end{eqnarray*}

Now, since  $u(x,T)=0$ over $\omega$ or $v(x,T)=0$ over $\omega$,   we have
that $u_\epsilon(x,T-\epsilon)=0$ over $\omega$ or $v_\epsilon(x,T-\epsilon)=0$ over $\omega$.
Thus, with the aid of the above $(i)$ and $(ii)$, we can use the result obtained in Step 1 (where $(u,v)$ is replaced by $(u_{\epsilon},v_{\epsilon})$) to get
  $u_\epsilon=v_\epsilon=0$ over $\Omega\times(0, T-\epsilon]$,
  i.e.,  $(u,v)=0$ over $\Omega\times(\epsilon, T]$.  Since $\epsilon$ can be arbitrarily taken from  $(0, T)$, we find that $(u,v)=0$ over $\Omega\times(0, T]$. This, along with the continuity of $u$, leads to  \eqref{4.17}.
  This  ends the proof.
\end{proof}


\end{document}